\documentclass[11pt]{amsart}

\usepackage[english]{babel}

\usepackage[colorinlistoftodos]{todonotes}

\usepackage[pdftex,paper=a4paper,portrait=true,textwidth=450pt,textheight=675pt,tmargin=3cm,marginratio=1:1]{geometry}

\usepackage{amsfonts}

\usepackage{amsmath}

\usepackage{bbm}

\usepackage{amsthm}

\usepackage{amssymb}
\usepackage{enumerate}

\usepackage{eufrak}

\usepackage{cancel}

\usepackage{color}

\usepackage[curve]{xypic}

\usepackage[hidelinks]{hyperref}

\usepackage{tikz-cd}

\usepackage{enumitem}

\usepackage{anyfontsize}

\usepackage{caption}

\newtheorem{theorem}{Theorem}[section]

\newtheorem{proposition-definition}[theorem]{Proposition-Definition}

\newtheorem{proposition}[theorem]{Proposition}

\theoremstyle{definition}

\newtheorem{definition}[theorem]{Definition}
\newtheorem{remark}[theorem]{Remark}
\newtheorem{example}[theorem]{Example}

\newcommand{\cQ}{\mathcal{Q}}
\newcommand{\mL}{\mathcal{L}}
\newcommand{\mG}{\mathcal{G}}
\newcommand{\mT}{\mathbf{T}}
\newcommand{\mE}{\mathcal{E}}
\newcommand{\mH}{\mathcal{H}}
\newcommand{\mB}{\mathcal{B}}
\newcommand{\mF}{\mathcal{F}}
\newcommand{\mU}{\mathcal{U}}
\newcommand{\mX}{\mathcal{X}}
\newcommand{\mD}{\mathcal{D}}
\newcommand{\mO}{\mathcal{O}}
\newcommand{\mP}{\mathbb{P}}
\newcommand{\mZ}{\mathbb{Z}}

\newcommand{\mM}{\mathcal{M}}

\newcommand{\mC}{\mathbb{C}}
\newcommand{\mQ}{\mathbb{Q}}
\newcommand{\ds}{\displaystyle}
\newcommand{\cN}{\mathcal{N}}
\newcommand{\vect}[1]{\boldsymbol{#1}}
\newcommand{\im}{\text{Im}}
\newcommand{\Hom}{\text{Hom}}
\newcommand{\mq}{\mathrm{q}}

\newcommand{\nocontentsline}[3]{}
\newcommand{\tocless}[2]{\bgroup\let\addcontentsline=\nocontentsline#1{#2}\egroup}

\let\amsamp=&
\begin{document}
		\title{Toric Sheaves on Hirzebruch Orbifolds}
	\author{Weikun Wang}
	\address{Department of Mathematics, University of Maryland, College Park, MD 20742, USA}
	\email{\href{wwang888@umd.edu}{wwang888@umd.edu}}
	\maketitle
	\begin{abstract}
		We provide a stacky fan description of the total space of certain split vector bundles, as well as their projectivization, over toric Deligne-Mumford stacks. We then specialize to the case of Hirzebruch orbifold $\mH_{r}^{ab}$ obtained by projectivizing $\mO \oplus \mO(r)$ over the weighted projective line $\mP(a,b)$. Next, we give a combinatorial description of toric sheaves on $\mH_{r}^{ab}$ and investigate their basic properties. With fixed choice of polarization and a generating sheaf, we describe the fixed point locus of the moduli scheme of $\mu$-stable torsion free sheaves of rank $1$ and $2$ on $\mH_{r}^{ab}$. As an example, we obtain explicit formulas for generating functions of Euler characteristics of locally free sheaves of rank 2 on $\mP(1,2) \times \mP^1$.
	\end{abstract}

	\tableofcontents
	\section{Introduction}

There is a nice class of toric vector bundles and projective bundles over  toric varieties. They can be constructed from toric fans and hence are also toric varieties. This type of bundles has been well studied in \cite{CLS11}. Given a fan, one can construct the line bundle corresponding to a Cartier divisor by extending the fan. Consequently, every vector bundle that can be decomposed into line bundles and its projectivization can be constructed from a fan.

This construction can be naturally generalized to the toric Deligne-Mumford stacks. Such stacks can be described by a stacky fan as in \cite{BCS05}. In the first section, we show that certain types of vector bundles can be constructed from stacky fans. As an application, we first give a general fan description of the weighted projective stacks. Then we construct projective bundles over weighted projective lines $\mP(a,b)$ and describe the Hirzebruch stacks, denoted by $\mH_r^{ab}$. When gcd$(a,b)=1$, in which case $\mH_r^{ab}$ is an orbifold, the stacky fan can be drawn as below, where $s,t \in \mZ$ are chosen so that $r=sa+bt$.     

\begin{figure}[!h]
	$$
	\begin{tikzpicture}
	\filldraw[black!25!white] (0,0) -- (0,2) -- (-2,2) -- (-2,1);
	\filldraw[black!50!white] (0,0) -- (3,2) -- (0,2);
	\filldraw[black!5!white] (0,0) -- (-2,1) -- (-2,-1) -- (0,-1);
	\filldraw[black!20!white] (0,0) -- (3,2) -- (3,-1) -- (0,-1);
	\draw[thin,->] (-2.5,0) -- (3.5,0) node[anchor=west] {x};
	\draw[thin,->] (0,-1) -- (0,2.5) node[anchor=south] {y};
	\draw[thick,->] (0,0) -- (3,2);
	\draw[thick,->] (0,0) -- (0,1);
	\draw[thick,->] (0,0) -- (0,-1);
	\draw[thick,->] (0,0) -- (-2,1);
	\node[right] at (3,2) {$\rho_1=(b,s)$};
	\node[left] at (-2,1) {$\rho_3=(-a,t)$};
	\node[above] at (0,1) {$\rho_2=(0,1)$};
	\node[below] at (0,-1) {$\rho_4=(0,-1)$};
	\node at (1.5,1.5) {$\sigma_1$};
	\node at (-1.5,1.5) {$\sigma_2$};
	\node at (-1.5,-0.5) {$\sigma_3$};
	\node at (1.5,-0.5) {$\sigma_4$};
	\end{tikzpicture}
	$$     
	\caption{} \label{fig1} 
\end{figure}  
Note that the fiber of the Hirzebruch surface over $\mP^1$ is always $\mP^1$. But this is not true for Hirzebruch stacks, in which case only the fiber over a non-stacky point is $\mP^1$.

Let $X$ be a nonsingular toric variety of dimension $d$. A. A. Klyachko \cite{Kly90}, M. Perling \cite{Per04} and M. Kool \cite{Koo10} have given a combinatorial description of $\mT$-equivariant coherent sheaves on toric varieties. The idea is that every toric variety can be covered by affine $\mT$-equivariant subvarieties $U_\sigma \cong \mC^d$, corresponding to the maximal cones in the fan. Locally, a sheaf is described by families of vector spaces, called $\sigma$-families. Those $\sigma$-families agree on the intersection of cones and satisfy some gluing conditions. 

The above idea is generalized to smooth toric Deligne-Mumford stacks first by A. Gholampour, Y. Jiang and M. Kool in \cite{GJK17}. Such stacks are covered by open substacks $\mU_\sigma \cong [\mC^d/N(\sigma)]$ \cite[Proposition 4.1]{BCS05}. Hence locally, a $\mT$-equivariant sheaf corresponds to a module with both $X(\mT)$-grading and $X(N(\sigma))$-fine-grading. The local data of such a sheaf consists of families of vector spaces with fine-gradings, called $S$-families. To obtain a sheaf globally, the gluing conditions are imposed. In the case of weighted projective stacks $\mP(a,b,c)$, the gluing conditions are given explicitly in \cite{GJK17}.

In the second section, we give the gluing conditions for Hirzebruch orbifolds. To glue the local data for any two substacks $U_{\sigma_i}$ and $U_{\sigma_{i+1}}$, we pull back the local data to their stack theoretic intersection. Matching $S$-families over the intersection allows us to describe $\mT$-equivariant coherent sheaves on Hirzebruch orbifolds. Then we can study torsion free sheaves and locally free sheaves on $\mH_{r}^{ab}$ and construct the moduli spaces.

In the third section, we investigate some basic properties of $\mH_{r}^{ab}$ including its coarse moduli scheme and modified Hilbert polynomial. From F. Nironi's work \cite{Nir08}, we know that a modified version of Hilbert polynomial is needed to define the Gieseker stability for stacks. Let $\epsilon$ be the structure morphism from $\mH_{r}^{ab}$ to its coarse moduli scheme $\mathrm{H}$. With fixed polarization $L$ on $\mathrm{H}$ and generating sheaf $\mE$ on $\mH_{r}^{ab}$, we define the modified Hilbert polynomial for a sheaf $\mF$ as 
$$P_\mE (\mF, T) =\chi (\mH_r^{ab}, \mF \otimes \mE^\vee \otimes \epsilon^* L^T )$$ 
and the modified Euler characteristic as 
$$\chi_\mE (\mF)= P_\mE (\mF, 0)$$ 

In the last section, we consider the moduli scheme of Gieseker stable and $\mu$-stable torsion free sheaves of rank $1$ and $2$ on Hirzebruch orbifolds. Extending the work of \cite{Koo10}, we generalize the characteristic function and match the GIT stability with Gieseker stability. By lifting the action of the torus $\mT$ to the moduli scheme $\mM_{P_\mE}^{\mu s}$ \hyperref[4.1]{[Section \ref*{4.1}]}, we can describe explicitly the fixed point locus $(\mM_{P_\mE}^{\mu s})^T$ by the GIT quotient $\mM_{\vec{\chi}}^{\mu s}$ with gauge-fixed characteristic function $\vec{\chi}$ similar to \cite[Theorem 4.15]{Koo10}.

In the case of rank $1$, it leads to the counting of partitions, which generalizes L. G\"ottsche's result for nonsingular projective surface in \cite{Got90}. In the case of higher rank, we express the relation between generating functions of the moduli space of $\mu$-stable torsion free and locally free sheaves \hyperref[4.2]{[Section \ref*{4.2}]}, which generalizes L. G\"ottsche's result for Hirzebruch surfaces in \cite{Got99}.

\begin{theorem} \label{1} 
	Suppose $\normalfont{\text{gcd}}(a,b)=1$. Let $P_\mE$ be a choice of modified Hilbert polynomial of a reflexive sheaf  of rank $R$ on $\mH_{r}^{ab}$ and $\chi_\mE$ be the modified Euler characteristic. Then 
	$$\sum_{\chi_\mE \in \mZ} e(M_{\mH_r^{ab}}(R,c_1,\chi_\mE)) \mq^{\chi_\mE} = \prod_{k=1}^{\infty} \frac{\sum_{\chi_\mE \in \mZ} e(M_{\mH_r^{ab}}^{\text{vb}} (R,c_1,\chi_\mE)) \mq^{\chi_\mE}}{(1-\mq^{-ak})^{2R} (1-\mq^{-bk})^{2R} }  . $$
\end{theorem}

We compute the generating function $H_{c_1}^{\text{vb}}(\mq):= \sum e(M_{\mH_r^{ab}}^{\text{vb}} (2,c_1,\chi_\mE)) \mq^{\chi_\mE}$ for locally free sheaves over Hirzebruch orbifolds $\mH_{r}^{ab}$ with fixed generating sheaf $\mE$ and polarization $L$ given in \hyperref[34]{[Section \ref*{3.4}]}. Especially when $r=0$, we obtain an expression for the orbifold $\mP(a,b) \times \mP^1$, which is parallel to M. Kool's result for $\mP^1 \times \mP^1$ \cite[Corollary 2.3.4]{Koo10}.

\begin{theorem} \label{3} 
	Suppose $\normalfont{\text{gcd}}(a,b)=1$. Let $ f=(\frac{n}{2}+ 1)(m+C)$ where $C = a + b + ab - 1$. Let $p=\normalfont{\text{gcd}}(b,r)=b$ and $q=\normalfont{\text{gcd}}(a,r)=a$ as $r=0$. Then for fixed first Chern class $c_1(\mF)= \frac{m}{a}x +n y$ where $c_1(\mD_{\rho_1})=x, c_1(\mD_{\rho_2})=y$, $\mD_{\rho_i}$ is the divisor corresponding to the ray $\rho_i$, the generating function $H_{c_1}^{\text{vb}}(\mq)$ for the orbifold $\mP(a,b) \times \mP^1$ is
	$$  
	\Bigg( - \sum_{C_1 } + \sum_{	C_4 } + \sum_{	C_5 } 	+ 2 \sum_{C_6 } 	\Bigg) \mq^{f - \frac{1}{2} j i} +  \Bigg( 2 \sum_{C_2 } + 2\sum_{C_3 }  \Bigg)
	\mq^{f - \frac{1}{4}ij  + \frac{1}{4}jk - \frac{1}{4}kl - \frac{1}{4}li } 
	$$
	where
	\begin{align*}
	C_1=\{ & (i,j,k,l) \in \mZ^4: 2 \mid m+i, 2 \mid n+j, 2 \mid j-l, 2b \mid i-k, 2a \mid i+k ,  i = pqj, \\ & -j < l < j, 
	-pqj <k<pqj \},  \\
	C_2=\{ & (i,j,k,l) \in \mZ^4: 2 \mid m+i, 2 \mid n+j, 2 \mid j-l, 2b \mid i-k, 2a \mid i+k , \\ & -i < k < pql < i,
	-pqj < k , l < j \}, \\
	C_3=\{ & (i,j,k,l) \in \mZ^4: 2 \mid m+i, 2 \mid n+j, 2 \mid j-l, 2a \mid i-k, 2b \mid i+k ,  \\ & -i < k < pql < i,
	-pqj < k , l < j \},\\
	C_4=\{ & (i,j,k) \in \mZ^3: 2 \mid m+i, 2 \mid n+j, 2 \mid j+k,  b \mid i ,  -\frac{i}{pq}  < k < \frac{i}{pq} < j \}, \\
	C_5=\{ & (i,j,k) \in \mZ^3: 2 \mid m+i, 2 \mid n+j, 2 \mid j+k,  a \mid i ,  -\frac{i}{pq}  < k < \frac{i}{pq} < j \},\\
	C_6=\{ & (i,j,k) \in \mZ^3: 2 \mid m+i, 2 \mid n+j, 2a \mid i+k,  2b \mid i-k , \\ &  -pqj < k < pqj < i \}.\\
	\end{align*}
\end{theorem}

Moreover, in the case of $a=1,b=2$, we can get more explicit expressions \hyperref[412]{[Proposition \ref*{412}]}.

\begin{remark}
	In the case where gcd$(a,b) =d \neq 1$, our method still works, but we need to make the following modifications.
	Firstly, the inertia stack $I\mH_r^{ab}$ will have more than one 2-dimensional components, which will cause a slight change in the modified Euler characteristic. Secondly, the stacky fan in \hyperref[fig1]{Figure \ref*{fig1}} is changed and $b$, $-a$ are replaced by $\frac{b}{d}$  and $-\frac{a}{d}$ \eqref{1.3.2}. Hence we need to set $p=\text{gcd}(\frac{b}{d},r)$ and $q=\text{gcd}(\frac{a}{d},r)$.
\end{remark}

\tocless{\section*{Acknowledgement}}
The author would like to thank his advisor, Dr.\;Amin Gholampour, for his encouragement and invaluable guidance throughout this research.\\

\section{Toric stacks}

In this section, we will briefly review various definitions of stacky fans and their associated toric Deligne-Mumford stacks. Toric stacks were first introduced in \cite{BCS05} and later in \cite{FMN10}. The theory was further generalized in \cite{GS15} which encompasses all the notions of toric stacks before. In this paper, we will refer to \cite{BCS05} the notation of toric stacks most of the time, but use \cite{GS15} when constructing the vector bundles.

\begin{definition}\label{11}
	A \textit{stacky fan} \cite{BCS05} is a triple $(N,\Sigma,\beta:\mZ^n \to N)$ where
	\begin{itemize}
		\item $N$ is a finitely generated abelian group of rank $d$, not necessarily free.
		\item $\Sigma$ is a rational simplicial fan in $N_{\mQ}:=N \otimes_{\mZ} \mathbb{Q}$ with $n$ rays, denoted by $\rho_1,...,\rho_n$.
		\item $\beta:\mZ^n \rightarrow N$ is a homomorphism with finite cokernel such that $\beta(e_i) \otimes 1 \in N_\mQ$ is on the ray $\rho_i$ for $1 \leq i \leq n$.
	\end{itemize}
\end{definition}

Given a stacky fan, the way to construct its corresponding toric stack $[Z_\Sigma/G_\beta]$ is as follows:

The variety $Z_\Sigma$ is defined as $\mC^n-V(J_\Sigma)$ where
$J_\Sigma=\langle \prod_{\rho_i \not\subset \sigma} z_i \, | \, \sigma \in \Sigma \rangle$
is a reduced monomial ideal. Suppose $N$ is of rank $d$, then there exists a free resolution 
$ 0 \rightarrow \mZ^r \xrightarrow{Q} \mZ^{d+r} \rightarrow N \to 0$
of $N$. Let the matrix $B: \mZ^n \to \mZ^{d+r}$ be a lift of the map $\beta: \mZ^n \to N$. Define the dual group $\text{DG}(\beta)=(\mZ^{n+r})^\star /\im([B \, Q]^\star)$, where $(-)^\star$ is the dual $\Hom_\mZ(-,\mZ)$. Let $\beta^\vee: (\mZ^n)^\star \rightarrow \text{DG}(\beta)$ be the composition of the inclusion map $(\mZ^n)^\star \to (\mZ^{n+r})^\star$ and the quotient map $(\mZ^{n+r})^\star \to \text{DG}(\beta)$. By applying the functor $\Hom_\mZ(-,\mC^*)$ to $\beta^\vee$, we get a homomorphism $G_\beta :=\Hom_\mZ(DG(\beta),\mC^*) \to (\mC^*)^n$ which leaves $Z_\Sigma$ invariant.

The quotient stack $[Z_\Sigma/G_\beta]$ is called the \textit{toric Deligne-Mumford stack} associated to the stacky fan $\Sigma$.

\begin{definition}\label{12}
	A (non-strict) \textit{stacky fan} \cite{GS15} is a pair $(\Sigma, \beta: L \to N)$, where $\Sigma$ is a fan on the lattice $L$ and $N$ is a finitely generated abelian group.
\end{definition} 
\begin{remark}
	Since the fan is defined on $L$ instead of $N$, we are allowed to assume that $\beta$ is of not finite cokernel. Interested readers can read \cite{GS15} for more details. In our paper, we will only consider $\beta$ with the finite cokernel, in which case the construction of $G_\beta$ in \cite{GS15} essentially agrees with \cite{BCS05}.
\end{remark}   	

\begin{remark}
	The stacky fan defined in \hyperref[11]{Definition \ref*{11}} is a special case of \hyperref[12]{Definition \ref*{12}}. When $N$ is free, the toric stack arising from such a stacky fan is called a \textit{fantastack} in \cite{GS15}. When $N$ is not free, the toric stack can be realized as a closed substack of a fantastack, called the \textit{non-strict fantastack}.
\end{remark}

Let $\beta:L=\mZ^n \rightarrow N=\mZ^d$ be a homomorphism with the finite cokernel as in \hyperref[11]{Definition \ref*{11}}. Given a cone $\sigma \in \Sigma$ in $N$, set $\widehat{\sigma} = \text{cone }(\{e_i|\rho_i \in \sigma \})$
where $\{ e_i \}_{i=1}^{n}$ is the standard basis for $L$.  Define $\widehat{\Sigma}$ in $L$ as the fan generated by all the cones $\widehat{\sigma}$. Then the stack defined by a triple $(N,\Sigma,\beta:L \rightarrow N)$ \cite{GS15} is same as the stack defined by a pair $(\widehat{\Sigma}, \beta: L \to N)$ \cite{GS15}. Conversely, if the rays of $\widehat{\Sigma}$ in $L$ are $e_i$, the image of $\widehat{\Sigma}$ under $\beta$ is a stacky fan $\Sigma$ in $N$. Since these two definitions agree in the case of the fantastack, we will use them interchangeably when constructing vector bundles.

\subsection{Weighted Projective Stack} ~\\

Let $w_1,w_2,...,w_{n+1} \in \mZ_{>0}$. The weighted projective stack $\mP(w_1,...,w_{n+1})$ is the quotient stack $[\mC^{n+1} - \{0\} / \mC^*]$ where $\mu \in \mC^*$ acts by $\mu (x_1,...,x_{n+1})=(\mu^{w_1}x_1, ..., \mu^{w_{n+1}}x_{n+1})$. We will give a general description of the stacky fan for the weighted projective stack. Firstly, we assume gcd$(w_1,...,w_{n+1})=1$, which means $\mP(w_1,...,w_{n+1})$ is an orbifold and the lattice $N$ is free.

\begin{proposition} \label{16}
	Let $\normalfont{\text{gcd}}(w_i,...,w_{n+1})=\lambda_i$ for $1 \leq i \leq n$. 
	Suppose $\lambda_1=1$. Define the map $\beta: \mZ^{n+1} \to \mZ^n$ by the following $n \times (n+1)$ matrix $B$:

	$$ \begin{bmatrix}
	\ds \frac{\lambda_2}{\lambda_1} & b_{12} & \cdots & b_{1,i-1} & b_{1i} & b_{1,i+1} &\cdots & b_{1,n-1} & b_{1n} & b_{1,n+1}\\
	\vdots &  & \ddots &    & & & & & & \vdots \\
	0 & 0 &  \cdots &  \ds \frac{\lambda_{i}}{\lambda_{i-1}} &  b_{i-1,i} & b_{i-1,i+1} & \cdots & b_{i-1,n-1}  & b_{i-1,n} & b_{i-1,n+1} \\
	\vdots  &  &  &  &  &  & \ddots & & & \vdots \\
	0 & 0 & \cdots& 0 &  0 & 0 & \cdots & \ds \frac{\lambda_{n}}{\lambda_{n-1}}  &  b_{n-1,n} & b_{n-1,n+1} \\
	0  &0  & \cdots &0 & 0  & 0 & \cdots & 0 & \ds \frac{w_{n+1}}{\lambda_n}   &  \ds -\frac{w_{n}}{\lambda_n}
	\end{bmatrix}
	$$
	where $b_{ij}$ are chosen so that 
	\begin{equation} \label{1.1.1} \tag{1.1.1} 
	\frac{\lambda_{i+1}}{\lambda_{i}} w_{i} + \sum_{j=i+1}^{n+1} b_{ij} w_j =0 \text{ for } 1 \leq i \leq n-1, 
	\end{equation}  
	$$
	0 \leq b_{1i}, b_{2i}, \cdots , b_{ii} < \frac{\lambda_{i+1}}{\lambda_i} \text{ for } 2 \leq i \leq n-1. $$
	
	Each column represents a ray in the fan $\Sigma$. The maximal cones of the fan are given by any $n$ rays. Then the triple $(\mZ^n,\Sigma,\beta)$ corresponds to the weighted projective orbifold $\mP(w_1,...,w_{n+1})$.	
\end{proposition}
Note that the choice is not unique.

\begin{proof}
	The triple induces $\mP(w_1,...,w_{n+1})$ if the following two statements are true:
	\begin{itemize}
		\item $\text{DG}(\beta)=\mZ$.
		\item $\beta^\vee:\mZ^{n+1} \to \mZ$ is given by 
		$\begin{bmatrix}
		w_1 & w_2 & ... & w_{n+1}
		\end{bmatrix}$.
	\end{itemize}
	If $\text{DG}(\beta)=\mZ^{n+1}/\im(B^\star) \cong \mZ$, then $\begin{bmatrix}
	w_1 & w_2 & ... & w_{n+1}
	\end{bmatrix}$ spans the integer null space of the matrix $B$ because $b_{ij}$ are chosen to satisfy \eqref{1.1.1}. Let $B_i$ denote the minor of $B$ by removing the $i$th column. If we can show that $\text{gcd}(\det(B_1), ... , \det(B_{n+1}))=1$, then there exists a matrix $\begin{bmatrix}
	\vect{b} \\
	B 
	\end{bmatrix}$ with determinant $1$. Hence  $\mZ^{n+1}/\im(B^\star)=\mZ$.
	
	When $i=n,n+1$, we obtain two diagonal matrices and  $\det(B_{n+1})=w_{n+1},  \det(B_{n})=-w_n$.  For $1 \leq i \leq n-1$, we compute by induction that $\det(B_{i})=(-1)^{n+1-i} w_{i}$. Denote by 
	$$C_i=\begin{bmatrix}
	b_{i,i+1} & \cdots & b_{i,n-1}  & b_{i,n} & b_{i,n+1}\\
	\ds \frac{\lambda_{i+2}}{\lambda_{i+1}} & \cdots & b_{i+1,n-1}  & b_{i+1,n} & b_{i+1,n+1}\\
	&\ddots & & & \vdots \\
	0 & \cdots & \ds \frac{\lambda_{n}}{\lambda_{n-1}}  &  b_{n-1,n} & b_{n-1,n+1} \\
	0 & \cdots & 0 & \ds \frac{w_{n+1}}{\lambda_{n}}   &  \ds -\frac{w_{n}}{\lambda_{n}}
	\end{bmatrix}$$ 
	the bottom-right $(n-i+1) \times (n-i+1)$ submatrix of $B$, then $\det(B_i)=  \lambda_i \cdot \det(C_i)$.
	
	For $i=n-1$, because $\text{gcd}(w_n,w_{n+1})=\lambda_n$ and $\lambda_{n-1} | w_{n-1} $, integers $b_{n-1,n}$ and $b_{n-1,n+1}$ can be chosen so that
	$$ \det(C_{n-1})= -b_{n-1,n} \frac{w_{n}}{\lambda_{n}}  - b_{n-1,n+1} \frac{w_{n+1}}{\lambda_{n}}=\frac{w_{n-1}}{\lambda_{n-1}}. $$
	
	Suppose integers $b_{i,i+1},...,b_{i,n+1}$ are chosen so that $$ \det(C_i)= (-1)^{n-i} \sum_{j=i+1}^{n+1} b_{i,j} \frac{w_j}{\lambda_{i+1}}=  (-1)^{n+1-i} \frac{w_i}{\lambda_i},$$ then we can expand the matrix $C_{i-1}$ by the first column and get
	$$\begin{array}{ll}
	\det(C_{i-1}) & \ds = b_{i-1,i} \det(C_i) - \frac{\lambda_{i+1}}{\lambda_{i}} \det(C_i ')             \\[-2pt]
	& \ds  = (-1)^{n+1-i} b_{i-1,i} \frac{w_i}{\lambda_i} - (-1)^{n-i} 	\frac{\lambda_{i+1}}{\lambda_{i}} \sum_{j=i+1}^{n+1} b_{i-1,j} \frac{w_j}{\lambda_{i+1}} \\[-2pt]
	& \ds =(-1)^{n-i} \frac{w_{i-1}}{\lambda_{i-1}},
	\end{array}
	$$
	where $C_i '$ is the submatrix of $C_{i-1}$ by removing the first column and the second row.
	
	Now we get $\det(B_{i})=(-1)^{n-i} w_{i}$ and $\text{gcd}(\det(B_1), ... , \det(B_{n+1}))=1$. 
	
	If $ b_{ji} \geq \frac{\lambda_{i+1}}{\lambda_i}$ or $ b_{ji} < 0$, then we can left multiply an elementary matrix and the integer null space will be unchanged.
\end{proof}

\begin{example}
	Consider the stack $\mP(1,2,4,8)$. Since $\text{gcd}(2,4,8)=2$, $\text{gcd}(4,8)=4$, the matrix for $\beta: \mZ^4 \to \mZ^3$ will be
	$$\begin{bmatrix}
	2 & a & b & c\\
	0 & 2 & d & e\\
	0 & 0 & 2 & -1
	\end{bmatrix}$$
	such that $	4 +4d+ 8e=0, 2 +2a +4b+8c=0.$
	One of the solutions for this system is as follows:
	$$\begin{bmatrix}
	2 & 1 & 1 & -1\\
	0 & 2 & 1 & -1\\
	0 & 0 & 2 & -1
	\end{bmatrix}.$$
\end{example}

When $\lambda_1 \neq 1$, the lattice $N$ is not free and can be identified as $\mZ^n \oplus \mZ / \lambda_1 \mZ$. In this case, $\mP(w_1,...,w_{n+1})$ is a $\mu_{\lambda_1}$-banded gerbe over $ \mP(\frac{w_1}{\lambda_1},...,\frac{w_{n+1}}{\lambda_1})$, which is isomorphic to the root stack $ \sqrt[\lambda_1]{\mO_{\mP(\frac{w_1}{\lambda_1},...,\frac{w_{n+1}}{\lambda_1})} (1) /   \mP(\frac{w_1}{\lambda_1},...,\frac{w_{n+1}}{\lambda_1})  }. $

\begin{proposition} \label{18}
	Choose $c_1,...,c_{n+1}$ so that $\sum_{i=1}^{n+1} c_i \frac{w_i}{\lambda_1} \equiv 1 \mod \lambda_1.$
	Set $\vect{c}=([c_1],...,[c_{n+1}])$ where $[c_i]$ is the class of $c_i$ modulo $\lambda_1$. Let $B'$ the matrix  corresponding to $ \mP(\frac{w_1}{\lambda_1},...,\frac{w_{n+1}}{\lambda_1})$ as in \hyperref[16]{Proposition \ref*{16}}.
	Define the map $\beta: \mZ^{n+1} \to \mZ^n \oplus \mZ / \lambda_1 \mZ$ by 
	$B=\begin{bmatrix}
	B' \\
	\vect{c}
	\end{bmatrix}.$
	Then the triple $(\mZ^n,\Sigma,\beta)$ corresponds to the weighted projective  stack $\mP(w_1,...,w_{n+1})$.	   
\end{proposition}
\begin{proof}
	The $[B \, Q]$ matrix as in \cite{BCS05} is given by 
	$\begin{bmatrix}
	B' & \vect{0} \\
	\vect{c} & \lambda_1
	\end{bmatrix}$.
	Since $ \sum_{i=1}^{n+1} c_i \frac{w_i}{\lambda_1} \equiv 1 \mod \lambda_1$, the vector
	$\begin{bmatrix}
	w_1 & w_2 & \cdots & w_{n+1} & *
	\end{bmatrix}$
	spans the integer null space of the matrix $[B \, Q]$.
\end{proof}

\subsection{Vector Bundles} ~\\

In \cite{CLS11}, it mentions a class of toric morphisms that have a nice local structure. This can be naturally generalized to the morphisms of fantastacks. 

Let $N_1, N_2$ be free abelian groups. Denote the bases of $\mZ^{n_1}$ and $\mZ^{{n_2}}$ by $\{e_1,...,e_{n_1}\}$ and $\{e_{n_1+1},...,e_{n_1+n_2}\}$. By abuse of notation, we also assume the basis of $\mZ^{n_1+n_2}$ is $\{e_1,...,e_{n_1}, \allowbreak e_{n_1+1}, ...,e_{n_1+n_2}\}$. Consider the exact sequence of the fantastacks given by a commutative diagram
\begin{equation}\tag{1.2.1} \label{1.2.1}
\begin{tikzcd}[every arrow/.append style={shift left}]
0 \arrow{r} & \mZ^{n_1} \arrow{r}{(\text{Id},0)} \arrow{d}{\beta_1} & \mZ^{n_1} \oplus \mZ^{n_2} \arrow{r}{\text{pr}_2}  \arrow{d}{\beta} & \mZ^{n_2} \arrow{r} \arrow{d}{\beta_2}  & 0 \\  
0 \arrow{r} & N_1 \arrow{r}{f}  & N_1 \oplus N_2 \arrow{r}{\text{pr}_2}  & N_2 \arrow{r} \arrow{l}{g} & 0 
\end{tikzcd}
\end{equation}
such that the rows are exact and the column morphisms are of the finite cokernel. Suppose there exists a splitting morphism $g$ satisfying the following conditions:

\begin{enumerate}
	\item $A$ is a $\text{rk}N_1 \times \text{rk}N_2$ integer matrix such that 
	$$\beta(e_i)=\left\{\begin{array}{lll}
	f(\beta_1(e_i)) = & \begin{bmatrix}
	\beta_1(e_i) \\
	0
	\end{bmatrix}  & \text{if } 1 \leq i \leq n_1\\[1em]
	g(\beta_2(e_i)) = &  \begin{bmatrix}
	A \beta_2(e_i) \\
	\beta_2(e_i)
	\end{bmatrix}
	& \text{if } n_1+1 \leq i \leq n_1+n_2.
	\end{array}\right.$$
	\item Given cones $\sigma_1 \in \Sigma_1$ and $\sigma_2 \in \Sigma_2$, the sum $\sigma_1 + \sigma_2$ lies in $\Sigma$, and every cone of $\Sigma$ arises this way.    
\end{enumerate}

Then we say $(\Sigma, \beta: \mZ^{n_1+n_2} \to N_1 \oplus N_2)$ is \textit{globally split} by $(\Sigma_1, \beta_1: \mZ^{n_1} \to N_1)$ and $(\Sigma_2, \beta_2: \mZ^{n_2} \to N_2)$.

\begin{theorem}
	If $(\Sigma, \beta: \mZ^{n_1+n_2} \to N_1 \oplus N_2)$ is globally split by $(\Sigma_1, \beta_1: \mZ^{n_1} \to N_1)$ and $(\Sigma_2, \beta_2: \mZ^{n_2} \to N_2)$, then $\mathcal{X}_{\Sigma,\beta} \cong \mathcal{X}_{\Sigma_1,\beta_1} \times \mathcal{X}_{\Sigma_2,\beta_2}.$
\end{theorem} 

\begin{proof}
	Denote the matrices for $\beta_1$ and $\beta_2$ by  
	$$\begin{array}{c}
	B_1=\begin{bmatrix}
	\beta_1(e_1) & \beta_1(e_2) & \cdots & \beta_1(e_{n_1})
	\end{bmatrix},\\
	B_2=\begin{bmatrix}
	\beta_2(e_{n_1+1}) & \beta_2(e_{n_1+2}) & \cdots & \beta_2(e_{n_1+n_2})
	\end{bmatrix}.
	\end{array}$$ \\
	The matrix for $\beta$ is given by 
	$B=\begin{bmatrix}
	B_1 & AB_2 \\
	0   & B_2
	\end{bmatrix}.$
	It is not hard to show that $DG(\beta) \cong DG(\beta_1) \oplus DG(\beta_2)$ and $\beta^\vee \cong \beta_1^\vee \oplus \beta_2^\vee$, which implies $\alpha \cong \alpha_1 \times \alpha_2$, where $\alpha, \alpha_1$ and $\alpha_2$ are obtained by applying $\text{Hom}_\mZ (-,\mC^*)$ to $\beta^\vee, \beta_1^\vee, \beta_2^\vee$.
	
	It remains to show $Z_\Sigma=Z_{\Sigma_1} \times Z_{\Sigma_2}$. The $\mC$-valued points of $Z_\Sigma$ are $z \in \mC^{n_1+n_2}$ such that the cone generated by the set $\{\rho_i:z_i=0\}$, where $\rho_i$ is the cone generated by $b_i$ in $N_\mQ$, belongs to $\Sigma$. Since every cone of $\Sigma$ is the sum of cones in $\Sigma_1$ and $\Sigma_2$, the $\mC$-valued points of $Z_\Sigma$ are exactly the product of $\mC$-valued points of $Z_{\Sigma_1}$ and $Z_{\Sigma_2}$.
\end{proof}

\begin{example}
	Consider the following exact sequence of fantastacks
	$$\begin{tikzpicture}
	\draw[thick,->] (0,0) -- (1,0);
	\node[above] at (0.8,0.1) {$1$};
	\end{tikzpicture}
	\begin{tikzpicture}
	\fill[black!30!white] (0,0) rectangle (1,1);
	\draw[thick,->] (0,0) -- (0,1);
	\draw[thick,->] (0,0) -- (1,0);
	\node[right] at (-1,0.7) {$(0,1)$};
	\node[right] at (0.9,0.2) {$(1,0)$};
	\end{tikzpicture}\
	\begin{tikzpicture}
	\draw[thick,->] (0,0) -- (1,0);
	\node[above] at (0.8,0.1) {$1$};
	\end{tikzpicture}
	$$
	$$    
	\begin{tikzcd}[every arrow/.append style={shift left}]
	0 \arrow{r} & \mZ^1 \arrow{r} \arrow{d}{1} & \mZ^2 \arrow{r}  \arrow{d}{\begin{bmatrix} 1 \amsamp 2 \\ 0 \amsamp 2 \end{bmatrix}} & \mZ^1 \arrow{r} \arrow{d}{2}  & 0 \, \\  
	0 \arrow{r} & \mZ^1 \arrow{r}  & \mZ^2 \arrow{r}  & \mZ^1 \arrow{r}  & 0. 
	\end{tikzcd}$$
	It can be shown that $\mathcal{X}_{\Sigma,\beta} = [\mC^2/\mu_2] \cong \mC \times [\mC/\mu_2]= \mathcal{X}_{\Sigma_1,\beta_1} \times \mathcal{X}_{\Sigma_2,\beta_2}$.
\end{example}

\begin{remark} 
	The above exact sequence of fantastacks can be better understood if we draw the corresponding stacky fans defined in \cite{BCS05}. The morphism from the middle stacky fan to the right can be viewed as the projection of rays from the lattice $\mZ^2$ to $\mZ$,
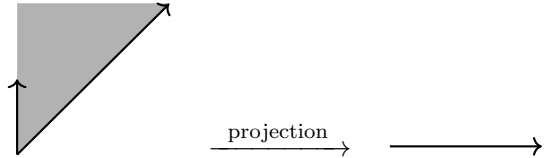
	$$
	\begin{tikzpicture}
	\filldraw[black!30!white] (0,0) -- (1.4,1.4) -- (0,1.4);
	\draw[thick,->] (0,0) -- (0,0.7);
	\draw[thick,->] (0,0) -- (1.4,1.4);
	\node[right] at (-1,0.7) {$(0,1)$};
	\node[right] at (1.4,1) {$(2,2)$};
	\end{tikzpicture}
	\xrightarrow{\text{projection}} \quad 
	\begin{tikzpicture}
	\draw[thick,->] (0,0) -- (1.4,0);
	\node[above] at (1.2,0.1) {$2$};
	\end{tikzpicture}	
	$$
	which is compatible with $\mathcal{X}_{\Sigma,\beta} \to  \mathcal{X}_{\Sigma_2,\beta_2}$ induced from the projection onto the second coordinate.	
\end{remark}

\begin{remark} \label{213}
	The morphism of stacky fans  below corresponds to a morphism of stacks $\mathcal{X}_{\Sigma,\beta} \to [\mC^1/\mu_2]$. Indeed, $\mathcal{X}_{\Sigma,\beta}$ is a line bundle over $[\mC^1/\mu_2]$ whose fiber over the stacky point corresponds to the non-trivial representation of $\mu_2$. Hence the stacky fan of $\mathcal{X}_{\Sigma,\beta}$ is not globally split.
	$$
	
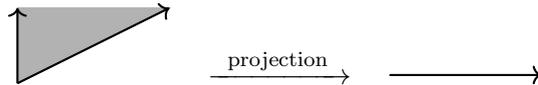
\begin{tikzpicture}
	\filldraw[black!30!white] (0,0) -- (1.4,0.7) -- (0,0.7);
	\draw[thick,->] (0,0) -- (0,0.7);
	\draw[thick,->] (0,0) -- (1.4,0.7);
	\node[right] at (-1,0.5) {$(0,1)$};
	\node[right] at (1.4,0.5) {$(2,1)$};
	\end{tikzpicture}
	\quad \xrightarrow{\text{projection}} \quad 
	\begin{tikzpicture}
	\draw[thick,->] (0,0) -- (1.4,0);
	\node[above] at (1.2,0.1) {$2$};
	\end{tikzpicture}	
	$$
\end{remark}

With the above theorem and examples in mind, we can generalize \cite[Definition 3.3.18]{CLS11}.

\begin{definition}
	Given an exact sequence like \eqref{1.2.1}, we say $(\Sigma, \beta: \mZ^{n_1+n_2} \to N_1 \oplus N_2)$ is \textit{(locally) split} by $(\Sigma_1, \beta_1: \mZ^{n_1} \to N_1)$ and $(\Sigma_2, \beta_2: \mZ^{n_2} \to N_2)$ if there exists a morphism $g: N_2 \to N_1 \oplus N_2$ satisfying the following conditions:
	
	\begin{enumerate}
		\item For every maximal cone $\sigma_j \in \Sigma_2$, there exists an  $\text{rk}N_1 \times \text{rk}N_2$ integer matrix $A_j$ such that 
		$$\beta(e_i)=\left\{\begin{array}{lll}
		f(\beta_1(e_i)) = & \begin{bmatrix}
		\beta_1(e_i) \\
		0
		\end{bmatrix}  & \text{if } 1 \leq i \leq n_1\\[1em]
		g(\beta_2(e_i)) = & \begin{bmatrix}
		A_j \beta_2(e_i) \\
		\beta_2(e_i)
		\end{bmatrix}
		& \text{if } e_i \in \sigma_j.
		\end{array}\right.$$
		\item Given cones $\sigma_1 \in \Sigma_1$ and $\sigma_2 \in \Sigma_2$, the sum $\sigma_1 + \sigma_2$ lies in $\Sigma$, and every cone of $\Sigma$ arises this way.    
	\end{enumerate}
\end{definition}

\begin{remark}
	The map $g$ here essentially gives the bijection $\sigma^\prime \to \hat{\sigma}$ for the case of toric varieties in \cite[Definition 3.3.18]{CLS11}.
\end{remark}

\begin{theorem}
	If $(\Sigma, \beta: \mZ^{n_1+n_2} \to N_1 \oplus N_2)$ is (locally) split by $(\Sigma_1, \beta_1: \mZ^{n_1} \to N_1)$ and $(\Sigma_2, \beta_2: \mZ^{n_2} \to N_2)$, then $\phi: \mathcal{X}_{\Sigma,\beta} \to \mathcal{X}_{\Sigma_2,\beta_2}$ is a locally trivial fiber bundle with fiber $\mathcal{X}_{\Sigma_1,\beta_1}$, i.e., $\mathcal{X}_{\Sigma_2,\beta_2}$ has a cover by affine open substacks $\mathcal{U}$ satisfying $\phi^{-1}(\mathcal{U}) \cong \mathcal{X}_{\Sigma_1,\beta_1} \times \mathcal{U}.$
\end{theorem} 

\begin{proof}
	The proof is similar to that of \cite[Theorem 3.3.19]{CLS11}.
\end{proof}
Therefore, if the stacky fan of a vector bundle is locally split, then for every stacky point of the base, the representation of the stabilizer group at that point on the fiber is trivial.

Note that the above theorem can be generalized to the case where $N_1$ and $N_2$ are not free.

\begin{example}\label{218}
	Consider the following morphism of stacky fans:
	$$
	
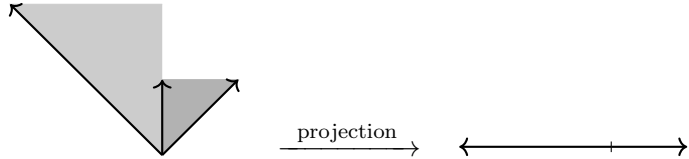
\begin{tikzpicture}
	\filldraw[black!30!white] (0,0) -- (0.7,0.7) -- (0,0.7);
	\filldraw[black!20!white] (0,0) -- (-1.4,1.4) -- (0,1.4);
	\draw[thick,->] (0,0) -- (0,0.7);
	\draw[thick,->] (0,0) -- (-1.4,1.4);
	\draw[thick,->] (0,0) -- (0.7,0.7);
	\node[right] at (0.7,0.5) {$(1,1)$};
	\node[right] at (0,0.9) {$(0,1)$};
	\node[right] at (-2.6,0.9) {$(-2,2)$};
	\end{tikzpicture}
	\quad \xrightarrow{\text{projection}} \quad 
	\begin{tikzpicture}
	\draw[thick,->] (0,0) -- (-1.4,0);
	\draw[thick,->] (0,0) -- (0.7,0);
	\draw (0 cm,2pt) -- (0 cm,-2pt);
	\node[above] at (-1.2,0.1) {$-2$};
	\node[above] at (0.6,0.1) {$1$};
	\end{tikzpicture}	
	$$	
	The induced morphism $\phi: \mathcal{X}_{\Sigma,\beta} \to \mP(2,1)$ corresponds to a line bundle such that its fan is locally split. But it cannot be written globally as the product of one-dimensional toric stacks. Indeed, it represents $\mO_{\mP(2,1)}(-4)$ by the next theorem.
\end{example}

For a vector bundle over a stack, the fiber over a stacky point might correspond to a non-trivial representation of the stabilizer group. In this case, the corresponding stacky fan is not locally split. To include this type of  stacky vector bundles, we generalize \cite[Sec. 7.3]{CLS11} to the case of toric stacks.

Let's assume $N$ is free. Given a triple $(N,\Sigma,\beta:\mZ^n \to N)$, we define the new stacky fan $(N \times \mZ,\widetilde{\Sigma},\widetilde{\beta}:\mZ^{n+1} \to N \times \mZ)$ as follows: 
\begin{enumerate}
	\item $\widetilde{\beta} (e_i)=(\beta(e_i),-a_i)$ for $1 \leq i \leq n$.
	\item $\widetilde{\beta} (e_{n+1})=(\vect{0},1)$.
	\item Given $\sigma \in \Sigma$, set $\widetilde{\sigma}=\text{Cone} \left((\vect{0},1), \, \widetilde{\beta} (e_i) \otimes 1 \, | \, \beta(e_i) \otimes 1 \in \sigma(1) \right)  \in N_\mQ \times \mQ$, and let $\widetilde{\Sigma}$ be the set consisting of  $\widetilde{\sigma}$ for all $\sigma \in \Sigma$ and their faces.
\end{enumerate}
The natural projection $\mZ^{n+1} \to \mZ^n$ is compatible with $\widetilde{\Sigma}$ and $\Sigma$. Therefore it gives a toric morphism
$\pi: \mX_{\widetilde{\Sigma},\widetilde{\beta}} \to \mX_{\Sigma,\beta}. $
\begin{theorem}
	Denote by $\mD_{\rho_i}$ the divisor corresponding to the ray $\rho_i$. Then $\pi: \mX_{\widetilde{\Sigma},\widetilde{\beta}} \to \mX_{\Sigma,\beta} $ is a line bundle whose sheaf of sections is $$\mO_{\mX_{\Sigma,\beta}}(\mD)=\mO_{\mX_{\Sigma,\beta}}(\sum_{i} a_i \mD_{\rho_i}).$$ 
\end{theorem} 

Recall that the category of locally free sheaves on $[Z/G]$ is equivalent to that of $G$-linearized locally free sheaves on $Z$. Without considering the equivariant structure, these $G$-linearized invertible sheaves are all isomorphic to the trivial sheaf $\mO_Z$. By the construction of a toric stack, $G$ can be thought of as a subgroup of $(\mC^*)^n$. Each $g=(\lambda_1,...,\lambda_n) \in G$ induces an isomorphism $\mO_Z \to g^* \mO_Z$ sending $1$ to $\lambda_i$. The sheaf $\mO_{\Sigma,\beta}(\mD_{\rho_i})$ has a $G$-invariant global section $z_i$ such that $g^* z_i =\lambda_i z_i$. \cite{BH06}

\begin{proof}[Proof of Theorem 1.17]
	We will use the definition of stacky fan from \cite{GS15}.	
	
	Given a triple $(N,\Sigma,\beta:\mZ^n \to N)$, we can construct the corresponding fan $\widehat{\Sigma}$ in $\mZ^n$, which corresponds to a toric variety $Z_{\widehat{\Sigma}}$. Then by \cite{CLS11}, we can construct a new fan $\widehat{\Sigma}^\prime \in \mQ^n \times \mQ$. Given $\widehat{\sigma} \in \widehat{\Sigma}$, set 
	$\widehat{\sigma}^\prime=\text{Cone} \left((\vect{0}, 1), (e_i, -a_i) | e_i \in \widehat{\sigma} \right)$
	and let $\widehat{\Sigma}^\prime$ be the set consisting of cones $\widehat{\sigma}^\prime$ for all $\widehat{\sigma} \in \widehat{\Sigma}$ and their faces. By \cite[Proposition 7.3.1]{CLS11}, $\pi:  Z_{\widehat{\Sigma}^\prime} \to Z_{\widehat{\Sigma}} $ is a line bundle whose sheaf of sections is $\mO_{Z_{\widehat{\Sigma}}}(\sum_{i} a_i D_{e_i})$ where $D_{e_i}$ is the divisor corresponding to the ray generated by $e_i$ in $\widehat{\Sigma}$. 
	
	It suffices to show that the $G_\beta$-linearizion of this bundle exists and  the action of $G_\beta$ on $Z_\Sigma$ can be lifted . Define $\widehat{\beta}^\prime:\mZ^n \times \mZ \to N \times \mZ$ by the following matrices
	$$\begin{bmatrix}
	\beta(e_1) & \cdots & \beta(e_n) & \vect{0} \\
	0          & \cdots & 0          & 1       
	\end{bmatrix}.$$
	Then $G_{\widehat{\beta}^\prime} \cong G_\beta$ and its action on the line bundle is compatible with the action of $G_\beta$ on $Z_\Sigma$. The toric stack $\mX_{\widehat{\Sigma}^\prime, \widehat{\beta}^\prime}$ defined by the stacky fan $(\widehat{\Sigma}^\prime, \widehat{\beta}^\prime:\mZ^n \times \mZ \to N \times \mZ)$ induces the above line bundle.
	
	However, the rays of $\widehat{\Sigma}^\prime$ do not form a standard basis. Hence $\mX_{\widehat{\Sigma}^\prime, \widehat{\beta}^\prime}$ is not a fantastack and it is not a stacky fan defined in \hyperref[11]{Definition \ref*{11}}.
	
	Consider the morphism of stacky fans given by the following commutative diagram:
	$$  
	\begin{tikzcd}[column sep=large]
	\widetilde{\Sigma} \arrow{r}  & \widehat{\Sigma}^\prime 
	\end{tikzcd}$$
	\vspace{-1em}
	$$
	\begin{tikzcd}[column sep=large]
	\mZ^n \times \mZ \arrow{r}{\alpha}  \arrow{d}{\widetilde{\beta}:= \widehat{\beta}^\prime \circ \alpha} & \mZ^n \times \mZ  \arrow{d}{\widehat{\beta}^\prime}   \\  
	N \times \mZ \arrow{r}{\cong}  & N \times \mZ  
	\end{tikzcd}$$
	where $\alpha$ is defined by the matrix
	$$\begin{bmatrix}
	& I_n & & & \vect{0} \\
	-a_1 & -a_2 & \cdots & -a_n & 1
	\end{bmatrix}
	$$
	and $I_n$ is the $n \times n$ identity matrix.
	Let $\widetilde{\sigma}= \text{Cone} \left( e_i|\alpha(e_i) \in \widehat{\Sigma}^\prime \right)$. The morphism satisfies the conditions mentioned in \cite[Theorem B.3]{GS15}. Thus  $\mX_{\widetilde{\Sigma}, \widetilde{\beta}} \to  \mX_{\widehat{\Sigma}^\prime, \widehat{\beta}^\prime}$ is an isomorphism and $\mX_{\widetilde{\Sigma}, \widetilde{\beta}}$ is a fantastack. The matrix of $\widetilde{\beta}$ is given by
	$$\begin{bmatrix}
	\beta(e_1) & \cdots & \beta(e_n) & \vect{0} \\
	-a_1         & \cdots & -a_n          & 1     
	\end{bmatrix}.$$
\end{proof}
\begin{example}\label{218}
	Consider the morphism of stacky fans as follows:
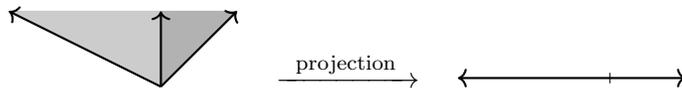
	$$
	\begin{tikzpicture}
	\filldraw[black!30!white] (0,0) -- (0.7,0.7) -- (0,0.7);
	\filldraw[black!20!white] (0,0) -- (-1.4,0.7) -- (0,0.7);
	\draw[thick,->] (0,0) -- (0,0.7);
	\draw[thick,->] (0,0) -- (0.7,0.7);
	\draw[thick,->] (0,0) -- (-1.4,0.7);
	\node[right] at (0.7,0.4) {$(1,1)$};
	\node[right] at (-0.6,0.9) {$(0,1)$};
	\node[right] at (-2.8,0.4) {$(-2,1)$};
	\end{tikzpicture}
	\quad \xrightarrow{\text{projection}} \quad 
	\begin{tikzpicture}
	\draw[thick,->] (0,0) -- (-1.4,0);
	\draw[thick,->] (0,0) -- (0.7,0);
	\draw (0 cm,2pt) -- (0 cm,-2pt);
	\node[above] at (-1.2,0.1) {$-2$};
	\node[above] at (0.6,0.1) {$1$};
	\end{tikzpicture}	
	$$	
	Then $\phi: \mathcal{X}_{\Sigma,\beta} \to \mP(2,1)$ is a line bundle whose sheaf of sections is $\mO_{\mP(2,1)}(-3)$ and its fan is not locally split.
\end{example}

Again this theorem can be generalized to the case where $N$ is not free.

\subsection{Projective Bundles} ~\\

Consider the locally free sheave 
$\mathcal{E}=\mO_{\mX_{\Sigma,\beta}}(\mD_0) \oplus \cdots \oplus \mO_{\mX_{\Sigma,\beta}}(\mD_r)$
given by the cartier divisors $\mD_i=\sum_{j=1}^n a_{ij} \mD_{\rho_j}$ for $0 \leq i \leq r$, then $\mP(\mathcal{E}) \to \mX_{\Sigma,\beta}$ is a projective bundle.

Assume $N$ is free. Given a triple $(N,\Sigma,\beta:\mZ^n \to N)$, we define the new stacky fan $(N \times \mZ^r,\widetilde{\Sigma},\widetilde{\beta}:\mZ^{n+r+1} \to N \times \mZ^r)$ as follows:

\begin{enumerate} 
	\item $\widetilde{\beta} (e_j)=(\beta(e_j),a_{1j} - a_{0j}, \cdots , a_{rj} - a_{0j} )$ for $1 \leq j \leq n$.
	\item $\widetilde{\beta} (e_{n+1+i})=(\vect{0},e_i) \in N \times \mZ^r$ for $0 \leq i \leq r$, where $e_0=-e_1-...-e_r \in \mZ^r$.
	\item Given $\sigma \in \Sigma$, set 
	$\widetilde{\sigma}_i=\text{Cone} \left(\widetilde{\beta} (e_j) \otimes 1 | \beta(e_j) \otimes 1 \in \sigma(1) \right) + 
	\text{Cone} \left((\vect{0},e_0), ... ,  (\vect{0},e_{i-1}), \right. \allowbreak \left. (\vect{0},e_{i+1}),... , (\vect{0},e_r)  \right) $
	and let $\widetilde{\Sigma}$ be the set consisting of cones $\widetilde{\sigma}_i$ for all $\sigma \in \Sigma$, $1 \leq i \leq r$ and their faces.
\end{enumerate}
Then the natural projection of the fan $\widetilde{\Sigma}$ induces a toric morphism
$\pi: \mX_{\widetilde{\Sigma},\widetilde{\beta}} \to \mX_{\Sigma,\beta}. $

\begin{theorem}
	$\mX_{\widetilde{\Sigma},\widetilde{\beta}}$ is the projective bundle $\mP(\mathcal{E})$.
\end{theorem}

\begin{proof}
	The proof is similar to that of \cite[Theorem 7.3.3]{CLS11}.
\end{proof}

Suppose gcd$(a,b)=1$, then by \hyperref[16]{Propostion \ref*{16}}, the fan of $\mP(a,b)$ is given by $\beta(e_1)=b$ and $\beta(e_2)=-a$ . Suppose $r=sa+tb$, then consider  
$$\mathcal{E} =\mO_{\mP(a,b)} \oplus \mO_{\mP(a,b)}(s\mD_{e_1}+t\mD_{e_2})
$$
where $\mD_{e_i}$ is the divisor corresponding to the ray generated by $\beta(e_i)$.
Hence $\widetilde{\beta}:\mZ^4 \to \mZ^2$ is given by 

\begin{equation}\tag{1.3.1} \label{1.3.1}
\begin{array}{ll}
\widetilde{\beta} (e_1)=(b,s), &
\widetilde{\beta} (e_2)=(-a,t), \\
\widetilde{\beta} (e_3)=(0,-1), &
\widetilde{\beta} (e_4)=(0,1).
\end{array}
\end{equation}

If gcd$(a,b)=d\neq 1$ and $c_1 \frac{a}{d} + c_2 \frac{b}{d} \equiv 1 \mod d$, then by \hyperref[18]{Proposition \ref*{18}}, the fan of $\mP(a,b)$ is given by $\beta':\mZ^2 \to \mZ \oplus \mZ/d\mZ$ such that 
$$
\beta' (e_1)=(\frac{b}{d}, c_1 \text{ mod } d),  \quad	
\beta' (e_2)=(-\frac{a}{d}, c_2 \text{ mod } d).
$$
Suppose $d \mid r$ and $r=sa+tb$, then for  $\mathcal{E} =\mO_{\mP(a,b)} \oplus \mO_{\mP(a,b)}(s\mD_{e_1}+t\mD_{e_2})$, 
$\widetilde{\beta}:\mZ^4 \to \mZ^2 \oplus \mZ/d\mZ$ is given by 

\begin{equation}\tag{1.3.2} \label{1.3.2}
\begin{array}{ll}
\ds \widetilde{\beta} (e_1)=(\frac{b}{d}, s , c_1 \text{ mod } d),	&
\ds \widetilde{\beta} (e_2)=(-\frac{a}{d}, t , c_2 \text{ mod } d), \\
\widetilde{\beta} (e_3)=(0,-1,0), &
\widetilde{\beta} (e_4)=(0,1,0).
\end{array}
\end{equation}

\begin{definition} \label{2.24}
	The \textit{Hirzebruch stack} $\mathcal{H}_r^{ab}$ is defined as $$\mathcal{H}_r^{ab}=\mP(\mO_{\mP(a,b)} \oplus \mO_{\mP(a,b)}(r))$$
	and its fan is given by \eqref{1.3.1} when gcd$(a,b)=1$ and by \eqref{1.3.2} when gcd$(a,b)=d \neq 1$. 
\end{definition}

From now on, to simplify the notation, we assume gcd$(a,b)=1$\footnote{Our method still works without this assumption.}. In this case, the matrix for $\beta: \mZ^4 \to \mZ^2$ is given by 
\begin{equation} \tag{1.3.3} \label{1.3.3}
B=\begin{bmatrix}
b & 0 & -a & 0 \\
s & 1 & t & -1
\end{bmatrix}
\end{equation}
where $r=sa+bt$. The stacky fan can be drawn as in \hyperref[fig2]{Figure \ref*{fig2}} and $\mathcal{H}_r^{ab}$ is called the \textit{Hirzebruch orbifold}.

\begin{figure}
	\centering
	\begin{tikzpicture}
	\filldraw[black!25!white] (0,0) -- (0,2) -- (-2,2) -- (-2,1);
	\filldraw[black!50!white] (0,0) -- (3,2) -- (0,2);
	\filldraw[black!5!white] (0,0) -- (-2,1) -- (-2,-1) -- (0,-1);
	\filldraw[black!20!white] (0,0) -- (3,2) -- (3,-1) -- (0,-1);
	\draw[thin,->] (-2.5,0) -- (3.5,0) node[anchor=west] {x};
	\draw[thin,->] (0,-1) -- (0,2.5) node[anchor=south] {y};
	\draw[thick,->] (0,0) -- (3,2);
	\draw[thick,->] (0,0) -- (0,1);
	\draw[thick,->] (0,0) -- (0,-1);
	\draw[thick,->] (0,0) -- (-2,1);
	\node[right] at (3,2) {$\rho_1=(b,s)$};
	\node[left] at (-2,1) {$\rho_3=(-a,t)$};
	\node[above] at (0,1) {$\rho_2=(0,1)$};
	\node[below] at (0,-1) {$\rho_4=(0,-1)$};
	\node at (1.5,1.5) {$\sigma_1$};
	\node at (-1.5,1.5) {$\sigma_2$};
	\node at (-1.5,-0.5) {$\sigma_3$};
	\node at (1.5,-0.5) {$\sigma_4$};
	\end{tikzpicture} 
	\caption{} \label{fig2}
\end{figure} 

\section{Sheaves on Hirzebruch Orbifolds} 

The Hirzebruch orbifold can be covered by open substacks of the form $[\mC^2/H]$ where $H$ is a finite abelian group and the actions of $H$ and the torus $\mT \cong (\mC^*)^2$ commute. Hence, to describe a $\mT$-equivariant sheaf on the Hirzebruch Orbifold, we define it locally over each substack and then glue each part together.

Let the character group $X(\mT) \cong \mZ^2$ be written additively and $m_i$ be the basis dual to the generators of rays of $\rho_i$. For $m \in X(\mT)$, we denote by $\chi(m): \mT \to \mC^*$ the actual character viewed as a function. 

Let $\mT$ act linearly on $\mC^2$. The action is given by $t \cdot x_i = \chi(m_i)(t)(x_i)$. Given a $\mT$-equivariant sheaf $\mF$ on $[\mathbb{C}^2/H]$, the corresponding module can be decomposed into $X(\mT)$-graded weight spaces:
$$H^0(\mC^2,\mF)=\bigoplus_{m \in X(\mT)} F(m).$$
Suppose $H$ acts by $h \cdot x_i = \chi(n_i)(h)(x_i)$, then $F(m)$ can be further decomposed into $X(H)$-graded weight spaces:
$$F(m)=\bigoplus_{n \in X(H)} F(m)_n.$$

Hence the category of $\mT$-equivariant sheaves on $[\mathbb{C}^2/H]$, by \cite{GJK17}, is equivalent to the category of stacky $S$-families. A object $\hat{F}$ in this category consists of the following data:
\begin{itemize}
	\item A collection of vector spaces $\{F(m)_n\}_{m \in X(\mT), n\in X(H)}$. 
	\item A collection of linear maps 
	$$\{\chi_i(m):F(m) \rightarrow F(m+m_i)\}_{i=1,2, m \in X(\mT)}.$$ 
	induced by multiplication by $x_i$ satisfying 
	$$\chi_i(m):F(m)_n \rightarrow F(m+m_i)_{n+n_i}, \chi_j(m+m_i) \cdot \chi_i(m)=\chi_i(m+m_j) \cdot \chi_j(m)$$
	for $i,j=1,2$, $m \in X(\mT)$ and $n \in X(H)$. 
\end{itemize}

\subsection{Open Affine Covers} ~\\

Let $N_{\sigma_i}$ be the subgroup of $N \cong \mZ^2$ generated by the rays of $\sigma_i$ and $N(\sigma_i)$ be the quotient group $N/N_{\sigma_i}$. By \cite{BCS05}, each $2$-dimensional cone $\sigma_i$ defines an open substack $\mU_i \cong [\mC^2/N(\sigma_i)]$. One can show that
$$\mU_1 \cong \mU_4 \cong [\mC^2/(\mZ/b\mZ)], \quad \mU_2 \cong \mU_3 \cong [\mC^2/(\mZ/a\mZ)]$$
and they form an open cover of $\mathcal{H}_r^{ab}$.

The integer null space of the matrix $B$ \eqref{1.3.3} is spanned by $\begin{bmatrix}
a & 0 & b & r
\end{bmatrix}$ and $\begin{bmatrix}
0 & 1 & 0 & 1
\end{bmatrix}$. Hence $(\tau, \lambda) \in G_\beta \cong (\mC^*)^2$ acts on $Z_\Sigma=\text{Spec }\mC[x,y,z,w] - V(xy,yz,zw,wx)$ by 
$$(\tau, \lambda) : (x,y,z,w) \to (\tau^a x, \lambda y, \tau^b z, \tau^r \lambda w)
$$
and $\mathcal{H}_r^{ab}=[Z_\Sigma/G_\beta]$.

Let $\beta_1$ be the morphism given by the first two columns of the matrix $B$. It induces a stacky fan with two rays and the corresponding toric stack $[Z_1/G_1]$ is exactly $ [\mC^2/(\mZ/b\mZ)]$. Consider the open subvariety $U_1$ of $Z_\Sigma$ defined as the complement of the vanishing locus of the monomial $zw$. There is a natural closed embedding $\phi_1: Z_1 \to U_1$ given by 
$$\phi_1(Z_1) = \mC^2 \times \vect{1} = \{(x,y,1,1)\} \in \mC^2 \times (\mC^*)^2 \cong U_1.$$

By \cite{BCS05}, an element $g \in G_\beta$ belongs to $G_1$ if and only if $\phi_1(Z_1) \cdot g \cap \phi_1(Z_1) \neq \emptyset$. In this case,
$$\tau^b=1, \tau^r \lambda=1 \Longrightarrow \lambda=\tau^{-r}.$$
Let $\mu_b$ be the group of $b$th roots of unity, then 
$$\mU_1 \cong [\mC^2/\mu_b], \quad  \tau \in \mu_b: (x,y) \to (\tau^a x, \tau^{-r} y).$$ 
Similarly, one can show that 
$$\begin{array}{ll}
\mU_2 \cong [\mC^2/\mu_a],  & \tau \in \mu_a: (y,z) \to (\tau^{-r} y, \tau^b z ),\\
\mU_3 \cong [\mC^2/\mu_a], & \tau\in \mu_a : (z,w)  \to (\tau^b z, \tau^r w),\\
\mU_4 \cong [\mC^2/\mu_b], & \tau\in \mu_b: (w,x) \to (\tau^r w, \tau^a x).
\end{array}$$

Consider the morphism $\widetilde{\phi_i}: \mU_i \hookrightarrow \mathcal{H}_r^{ab}$ induced by $ Z_i \xrightarrow{\phi_i} U_i \hookrightarrow Z_\Sigma$. We can compute stack theoretic intersections via the fiber product of $U_i$ and $U_j$ over $\mathcal{H}_r^{ab}$.

$$\begin{tikzcd}[column sep=large]
\mU_{12} := \mU_1 \times_{\mathcal{H}_r^{ab}} \mU_2 \arrow{r}  \arrow{d} & \left[\text{spec } \mC[x,y]/\mu_b\right]  \arrow{d}{\widetilde{\phi_1}}   \\  
\left[\text{spec } \mC[y,z]/\mu_a\right]  \arrow{r}{\widetilde{\phi_2}}  & \mathcal{H}_r^{ab} 
\end{tikzcd}$$
By calculating the fiber product of the corresponding groupoids \cite{ALR07}, one can show that
$$\mU_{12} \cong [\mC \times \mC^* / \mu_b \times \mu_a], \quad
(\mu,\nu) \in \mu_b \times \mu_a: (y,\tau) \to (\mu^{-r}y, \nu \mu^{-1} \tau).$$

Similarly, the fiber products of other open substacks are given as follows:
	$$\begin{array}{ll}
\mU_{23}   \cong [\mC   \times   \mC^*   \times   \mu_a / \mu_a   \times   \mu_a], 
&  (\mu,\nu) \in \mu_a   \times   \mu_a: (z, \lambda, \tau)   \to   (\mu^b z, \mu^r \lambda , \nu \mu^{-1} \tau).\\
\mU_{34} \cong [\mC \times \mC^* / \mu_a   \times   \mu_b], 
&   (\mu,\nu) \in \mu_a   \times   \mu_b: (w,\tau) \to (\mu^r w, \nu \mu^{-1} \tau).\\
\mU_{41}   \cong [\mC   \times   \mC^*   \times   \mu_b / \mu_b   \times   \mu_b], 
&   (\mu,\nu) \in \mu_b   \times   \mu_b: (x, \lambda, \tau)   \to   (\mu^a x, \nu^{-r} \lambda , \nu \mu^{-1} \tau).
\end{array}$$
Actually $\mU_{23}$ can be further simplified. Consider the groupoid morphism 
$$(\psi_1 \times \psi_0 , \psi_0) : (\mu_a \times \mC \times \mC^* \, \substack{\rightarrow\\[-1em] \rightarrow} \, \mC \times \mC^*) \longrightarrow (\mu_a \times \mu_a \times \mC \times \mC^* \times \mu_a \, \substack{\rightarrow\\[-1em] \rightarrow} \, \mC \times \mC^* \times \mu_a)  $$
defined by 	$$\psi_1(\mu) = (\mu, \mu), \quad \psi_0 ( z , \lambda) = (z , \lambda , 1).$$
One can show that it is a Morita equivalence and hence  
$$\mU_{23} \cong [\mC \times \mC^* / \mu_a ], \quad  \mu \in \mu_a: (z, \lambda) \to (\mu^b z, \mu^r \lambda).$$
Similarly,
$$\mU_{41} \cong [\mC \times \mC^* / \mu_b ], \quad  \mu \in \mu_b: (x, \lambda) \to (\mu^a x, \mu^{-r} \lambda).$$

The open immersions $\widetilde{\phi}_{ij}: \mU_{ij} = [Z_{ij}/G_{ij}] \hookrightarrow \mU_i= [Z_i/G_i]$ and $\widetilde{\phi}_{ji}: \mU_{ij} \hookrightarrow  \mU_j$ are induced from $\phi_{ij}: Z_{ij} \to Z_i$ and $\phi_{ji}: Z_{ji}=Z_{ij} \to Z_j$.
$$\begin{array}{ll}
\phi_{12}: (y,\tau)   \to (\tau^{-a},y)  \quad &              \phi_{21}: (y,\tau)   \to (y \tau^{-r}, \tau^b).  \\
\phi_{23}: (z, \lambda)  \to (\lambda^{-1}, z)  \quad & \phi_{32}: (z, \lambda)  \to (z, \lambda ).  \\
\phi_{34}: (w,\tau)  \to (\tau^{-b}, w)  \quad &                \phi_{43}: (w,\tau)  \to (\tau^{r} w,  \tau^a). \\
\phi_{41}: (x, \lambda)  \to (\lambda^{-1}, x)  \quad & \phi_{14}: (x, \lambda)  \to (x, \lambda ).  
\end{array}$$

To find $X(T)$-grading on each open substack $\mU_i$, we need to determine how the torus $\mT$ is embedded in $\mathcal{H}_r^{ab}$. One can show that
$$\begin{array}{c}
\mU_{1234} := \mU_{12} \times_{\mathcal{H}_r^{ab}} \mU_{34} \cong [(\mC^*)^2/\mu_b \times \mu_a] \\
(\mu,\mu') \in \mu_b \times \mu_a : (\alpha,\beta) \to (\mu (\mu')^{-1} \alpha, \mu^{-r} \beta ).
\end{array} $$
Hence if gcd$(a,b)=1$, then $\mU_{1234} \cong (\mC^*)^2$.
Suppose $(\mC^*)^2$ acts on itself by multiplication, then we can extend this action to the orbifold $\mH_r^{ab}$ by requiring all the open immersions to be $\mT$-equivariant.

For example, from the following commutative diagram
$$\begin{tikzcd}[row sep=large]
Z_{1234} \arrow{r} &  Z_{12} \cong \mC \times \mC^*  \arrow{r} \arrow[rr, bend left=30, looseness=0.2] & Z_1=\text{spec } \mC[x,y]  & Z_2=\text{spec } \mC[y,z] \\ [-15pt]
(\alpha, \beta)	 \arrow{r}  \arrow{d}{(0,1)}[swap]{(1,0)}        & (\beta ,\alpha^{-1}) \arrow{r}  \arrow{d}{(-1,0)}[swap]{(0,1)}  \arrow[rr, bend left=30, looseness=0.3] & (\alpha^a, \beta) \arrow{d}{(0,1)}[swap]{(a,0)}  &    (\beta \alpha^r, \alpha^{-b}) \arrow{d}{(-b,0)}[swap]{(r,1)}  \\ [-5pt] 
(t_1 \alpha, t_2 \beta)	\arrow{r}        & (t_2 \beta, t_1^{-1}\alpha^{-1} )   \arrow{r} \arrow[rr, bend right=30, looseness=0.3] & (t_1^a \alpha^a, t_2 y) &   (t_2 \beta t_1^r \alpha^r , t_1^{-b} \alpha^{-b})
\end{tikzcd}$$
we see that $\mT$-weights are $(0,1)$ and $(-1,0)$ on $Z_{12}$,  $(a,0)$ and $(0,1)$ on $Z_1$, $(r,1)$ and $(-b,0)$ on $Z_2$.

Similarly, one can show that $\mT$-weights are given by the following tables:
$$\begin{tabular}{|c|c|} 
\hline
& $\mT$-weights on $Z_i$ \\ \hline
$\mU_1$ & $(a,0),(0,1)$\\\hline
$\mU_2$ & $(r,1), (-b,0)$\\\hline
$\mU_3$ & $(-b,0),(-r,-1)$\\\hline
$\mU_4$ & $(0,-1),(a,0)$ \\ \hline
\end{tabular}
\hspace{50pt}
\begin{tabular}{|c|c|}
\hline
& $\mT$-weights on $Z_{ij}$ \\ \hline
$\mU_{12}$ & $(0,1),(-1,0)$\\\hline
$\mU_{23}$ & $(-b,0), (-r,-1) $\\\hline
$\mU_{34}$ & $(-r,-1),(1,0)$\\\hline
$\mU_{41}$ & $(a,0), (0,1)$ \\ \hline
\end{tabular}
$$

\subsection{Gluing Conditions} ~\\

To describe $\mT$-equivariant  torsion free sheaves on $\mH_r^{ab}$, we first determine the stacky $S$-family $\hat{F}_i$ of the sheaf $\mF_i$ on each open $\mU_i$. Then we pull back those families to the intersection $\mU_{ij}$ and match them for all $i,j$. This allows us to glue those sheaves $\mF_i$ to get a sheaf $\mF$ on $\mathcal{H}_r^{ab}$. Note that this gluing approach follows closely the work of \cite{GJK17}.

Let's first compute the family $\hat{F}_{1,12}$, which is the pullback of $\hat{F}_1$. 

Given a torus action $t \cdot x_i = \chi(m_i)(t)(x_i)$, the associated box $\mB_\mT$ \cite{GJK17} is defined as the subset of $X(\mT)$ of elements of the form $\sum_i q_i m_i$ with $0 \leq q_i < 1$. By the above table, the $\mT$-weights on $U_1$ are $(a,0)$ and $(0,1)$. Hence $q_1=\frac{k}{a}$ for $0 \leq k \leq a-1$ and $q_2=0$. Note that the box $\mB_\mT$ of $\mU_1$ can also be viewed as $[0,a-1] \times 0$ and the size of this box is $a$. 

For the stacky $S$-family $\hat{F}_1$, denote by 
$$ _{(k/a, 0)} F_1(l_1,l_2)$$
the vector space whose $\mT$-weight is $(k/a+l_1)(a,0)+(0+l_2)(0,1)$.

Consider the inclusion $\mU_{12} \hookrightarrow \mU_1$ induced from 
$$
\mC \times \mC^* \to \mC^2, \quad
\phi_{12}: (y,\tau)   \to (\tau^{-a}, y ) = (x,y). 
$$ 
We first restrict the sheaf $\mF_1$ to $\text{Im}(\phi_{12}) \cong \mC^* \times \mC$  and then pull it back along the morphism $\phi_{12}$. 

The sheaf $\mF$ is torsion free, hence the vector spaces $_{(k/a, 0)} F_1(l_1,l_2)$ stabilize for $l_1 \gg 0$, $l_2$ fixed. It means that they are isomorphic for $l_1 \gg 0$ \cite{Koo11}. We denote this limit by 
$$_{(k/a, 0)} F_1(\infty,l_2).$$
The sheaf $\mF_1|_{\mC^* \times \mC}$ corresponds to a $S$-family $\hat{G}_1$ and
$$_{(k/a, 0)} G(l_1,l_2) = \, _{(k/a, 0)} F_1(\infty,l_2)$$
is independent of $l_1$ because $G_1$ is a $\mC[x^{\pm}, y]$-module and multiplication by $x$ induces an isomorphism of vector spaces.

Pulling back the family $\hat{G_1}$ to $Z_{12}$ along the \'etale morphism $\phi_{12}$, we get a $\mC[\tau^{\pm}, y]$-module. An element of $\hat{F}_{1,12}$ at the weight $(k/a+l_1)(a,0)+(0+l_2)(0,1)$ can be uniquely written as
$$ \bigoplus_{0 \leq k' \leq a-1} v_{k'} \otimes \tau^{k'-k} $$ 
where $v_{k'} \in \, _{(k'/a, 0)} G_1(l_1,l_2)$, since the $\mT$-weight of $\tau$ is $(-1,0)$ on $U_{12}$.

Next, we set the fine-grading on the limit space $_{(k/a, 0)} F_1(\infty,l_2)$ by 
$$_{(k/a, 0)} F_1(\infty,l_2)_m = \,  _{(k/a, 0)} G_1(0,l_2)_m.$$
Thus the fine-grading of $S$-family $\hat{G_1}$ for any $l_1$ will be  
$$_{(k/a, 0)} G_1(l_1,l_2)_m= \, _{(k/a, 0)} G_1(0,l_2)_{m-al_1} \otimes \hat{\mu}_b^{al_1}.$$
Here  $\otimes \hat{\mu}_b$ means tensoring with the $1$-dimensional representation of the group $\mu_b$ of weight $1 \in \mZ/b\mZ$.

Since the $\mu_b \times \mu_a$-weight of $\tau$ is $(-1,1)$ on $\mU_{12}$,  the $S$-family of $\hat{F}_{1,12}$ at the $\mT$-weight $(k/a+l_1)(a,0)+(0+l_2)(0,1)$
with the fine grading is 
$$\begin{array}{ll}
&\ds \bigoplus_{\substack{ 0 \leq k' \leq a-1 \\ 
		m \in \mathbb{Z}/b\mathbb{Z}   }} \, _{(k'/a, 0)} G_1(l_1, l_2)_m \otimes \hat{\mu}_b^{k-k'} \otimes \hat{\mu}_a^{k'-k} \\[2em]
=  &  \ds \bigoplus_{\substack{ 0 \leq k' \leq a-1 \\
		m \in \mathbb{Z}/b\mathbb{Z} }}  \,  _{(k'/a, 0)} F_1(\infty, l_2)_m \otimes \hat{\mu}_b^{k-k'+al_1} \otimes \hat{\mu}_a^{k'-k}.
\end{array}
$$ 
Similarly, one can show that the $S$-family of $\hat{F}_{2,12}$ at the $\mT$-weight $(0+l_1)(r,-r)+(j/b+l_2)(-b,0)$
is 
$$ \bigoplus_{\substack{0 \leq j' \leq b-1\\
		n \in \mathbb{Z}/a\mathbb{Z}  }}
\,  _{(0, j'/b)} F_2(l_1, \infty)_n \otimes \hat{\mu}_a^{j-j'} \otimes \hat{\mu}_b^{j'-j+bl_2}.
$$ 
Since multiplication by $\tau$ is an isomorphism, the $S$-family $\hat{F}_{1,12}$ is determined by its elements at the weight $(0/a+0)(a,0)+(0+l)(0,1)=(0, l)$ for all $l \in \mZ$.
Therefore it suffices to compute the $S$-family $\hat{F}_{1,12}$ at the above weight, which is given by
\[\bigoplus_{\substack{0 \leq k' \leq a-1\\
		m \in \mZ/b\mZ	} }
\,  _{(k'/a, 0)} F_1(\infty, l_2)_m \otimes \hat{\mu}_b^{-k'} \otimes \hat{\mu}_a^{k'}.\] 
Similarly, we only compute the $S$-family $\hat{F}_{2,12}$ at the weight $(0+l)(r,1)+(0/b+0)(-b,0) =(lr , l)$,
which is given by 
$$\bigoplus_{\substack{0 \leq j' \leq b-1\\
		n \in \mZ/a\mZ	} } \, 
_{(0, j'/b)} F_2(l,\infty)_n \otimes \hat{\mu}_a^{-j'} \otimes \hat{\mu}_b^{j'}.$$ 
We can't equate them since they are at different weights. To jump from the weight $(lr , l)$ to $ (0, l)$, we multiply the second family by $\tau^{lr}$ as the $\mT$-weight of $\tau$ is $(-1,0)$. As a result, the fine grading is changed to 
$$\bigoplus_{\substack{0 \leq j' \leq b-1\\
		n \in \mZ/a\mZ}}
\,  _{(0, j'/b)} F_2(l,\infty)_n \otimes \hat{\mu}_a^{-j'+lr} \otimes \hat{\mu}_b^{j'-lr}.$$ 
Hence the gluing conditions on the substack $\mU_{12}$ are given by: 
$$\bigoplus_{\substack{ 0 \leq k \leq a-1 \\ m \in \mZ/b\mZ} }  \,  _{(k/a, 0)} F_1(\infty, l)_m \otimes \hat{\mu}_b^{-k} \otimes \hat{\mu}_a^{k} \cong  \bigoplus_{ \substack{0 \leq j \leq b-1 \\ n \in \mZ/a\mZ} }  \,  _{(0, j/b)} F_2(l,\infty)_n \otimes \hat{\mu}_a^{-j+lr} \otimes \hat{\mu}_b^{j-lr}$$ 
for all $l \in \mZ$. Here  $\otimes \hat{\mu}_b^k$ means tensoring with the $1$-dimensional representation of the group $\mu_b$ of weight $k \in \mZ/b\mZ$ and $\otimes \hat{\mu}_a^j$ means tensoring with the $1$-dimensional representation of the group $\mu_a$ of weight $j \in \mZ/a\mZ$. 

Similarly, we can get gluing conditions for other substacks. 

\begin{proposition}
	The category of $\mT$-equivariant torsion free sheaves on the Hirzebruch orbifold $\mathcal{H}_r^{ab}$ is equivalent to the category of finite \cite[Definition 5.10]{Per04} stacky $S$-families $\{ \hat{F}_i \}_{i=1,2,3,4}$ on $\mU_i$ satisfying the gluing conditions given by the following equalities of $\mu_a \times \mu_b$ representations:
	$$
	\begin{array}{rl}
	 \ds \bigoplus_{ \substack{0 \leq k \leq a-1\\
			m \in \mZ/b\mZ}} 
	\,  _{(k/a, 0)} F_1(\infty, l)_m \otimes \hat{\mu}_b^{-k} \otimes \hat{\mu}_a^{k} &   \cong   
	\ds \bigoplus_{ \substack{0 \leq j \leq b-1 \\
			n \in \mZ/a\mZ} }  
	\,  _{(0, j/b)} F_2(l,\infty)_n \otimes \hat{\mu}_a^{-j+lr} \otimes \hat{\mu}_b^{j-lr} \\[2.2em]
	
	\ds \bigoplus_{m \in \mZ/a\mZ}  \,  \: \:   _{(0, j'/b)} F_2(\infty, l)_m  &        \cong  \: \: \: \ds \bigoplus_{n \in \mZ/a\mZ}  \, \: \: \:   _{(j'/b, 0)} F_3(l, \infty)_n  \\[1.5em]
	
	 \ds \bigoplus_{ \substack{0 \leq j \leq b-1\\
			m \in \mZ/a\mZ} }  
	\,  _{(j/b, 0)} F_3(\infty, l)_m \otimes \hat{\mu}_a^{-j} \otimes \hat{\mu}_b^{j} &       \cong
	\ds  \bigoplus_{ \substack{0 \leq k \leq a-1 \\
			n \in \mZ/b\mZ} }  
	\,  _{(0, k/a)} F_4(l,\infty)_n \otimes \hat{\mu}_b^{-k-lr} \otimes \hat{\mu}_a^{k+lr} \\[2.2em]
	
	\ds \bigoplus_{m \in \mZ/b\mZ}  \,  \: \:   _{(0, k'/a)} F_4(\infty, l)_m &         \cong \: \: \:  \ds \bigoplus_{n \in \mZ/b\mZ}   \,   \: \: \:  _{(k'/a, 0)} F_1(l, \infty)_n  
\end{array}
$$
for all $l \in \mZ$, $j' \in \mZ/a\mZ$, $k' \in \mZ/b\mZ$ and similar gluing conditions between the corresponding inclusions.
\end{proposition}

\subsection{Examples} ~\\

In this section, we will give some examples of torsion free toric sheaves of rank $1$ and $2$ on $\mH_r^{ab}$.

\begin{example}
	Let $F$ be a  torsion free sheaf of rank $1$ on the Hirzebruch surface $\mH_r^{11}$. Then the gluing conditions are
	$$\begin{array}{ll}
	_{(0, 0)} F_1(\infty, l)  = \,  _{(0, 0)} F_2(l,\infty),  &
	_{(0, 0)} F_2(\infty, l)   = \,  _{(0, 0)} F_3(l, \infty),  \\
	_{(0, 0)} F_3(\infty, l)    =   \,  _{(0, 0)} F_4(l, \infty),  &
	_{(0, 0)} F_4(\infty, l)  =  \,  _{(0, 0)} F_1(l, \infty).
	\end{array}
	$$
	On each chart, $_{(0, 0)} \hat{F}_i$ can be described as follows:
	
	\begin{displaymath}
	\xy
	(0,0)*{} ; (30,0)*{} **\dir{} ; (0,5)*{} ; (30,5)*{} **\dir{.} ; (0,10)*{} ; (30,10)*{} **\dir{.} ; (0,15)*{} ; (30,15)*{} **\dir{.} ; (0,20)*{} ; (30,20)*{} **\dir{.} ; (0,25)*{} ; (30,25)*{} **\dir{.} ; (0,30)*{} ; (30,30)*{} **\dir{} ;                   (0,0)*{} ; (0,30)*{} **\dir{} ; (5,0)*{} ; (5,30)*{} **\dir{.} ; (10,0)*{} ; (10,30)*{} **\dir{.} ; (15,0)*{} ; (15,30)*{} **\dir{.} ; (20,0)*{} ; (20,30)*{} **\dir{.} ; (25,0)*{} ; (25,30)*{} **\dir{.} ; (30,0)*{} ; (30,30)*{} **\dir{} ; (5,30)*{} ; (5,15)*{} **\dir{-} ;  (5,15)*{} ; (20,15)*{} **\dir{-} ; (20,15)*{} ; (20,5)*{} **\dir{-} ; (20,5)*{} ; (30,5)*{} **\dir{-} ;  (-5,15)*{\hat{F}_1} ; (5,5)*{\bullet} ; (0,0)*{(A_1,A_2)} ;  (25,25)*{\mathbb{C}} ;                                                                                                        ; (50,0)*{} ; (80,0)*{} **\dir{} ; (50,5)*{} ; (80,5)*{} **\dir{.} ; (50,10)*{} ; (80,10)*{} **\dir{.} ; (50,15)*{} ; (80,15)*{} **\dir{.} ; (50,20)*{} ; (80,20)*{} **\dir{.} ; (50,25)*{} ; (80,25)*{} **\dir{.} ; (50,30)*{} ; (80,30)*{} **\dir{} ;                   (50,0)*{} ; (50,30)*{} **\dir{} ; (55,0)*{} ; (55,30)*{} **\dir{.} ; (60,0)*{} ; (60,30)*{} **\dir{.} ; (65,0)*{} ; (65,30)*{} **\dir{.} ; (70,0)*{} ; (70,30)*{} **\dir{.} ; (75,0)*{} ; (75,30)*{} **\dir{.} ; (80,0)*{} ; (80,30)*{} **\dir{} ;  (55,30)*{} ; (55,25)*{} **\dir{-} ;  (55,25)*{} ; (65,25)*{} **\dir{-} ; (65,25)*{} ; (65,15)*{} **\dir{-} ; (65,15)*{} ; (70,15)*{} **\dir{-} ; (70,15)*{} ; (70,5)*{} **\dir{-} ; (70,5)*{} ; (80,5)*{} **\dir{-} ;  (45,15)*{\hat{F}_2} ; (55,5)*{\bullet} ; (50,0)*{(A_2,A_3)} ;  (75,25)*{\mathbb{C}} ;                            
	\endxy 
	\end{displaymath}
	\begin{displaymath}
	\xy
	(0,0)*{} ; (30,0)*{} **\dir{} ; (0,5)*{} ; (30,5)*{} **\dir{.} ; (0,10)*{} ; (30,10)*{} **\dir{.} ; (0,15)*{} ; (30,15)*{} **\dir{.} ; (0,20)*{} ; (30,20)*{} **\dir{.} ; (0,25)*{} ; (30,25)*{} **\dir{.} ; (0,30)*{} ; (30,30)*{} **\dir{} ;                   (0,0)*{} ; (0,30)*{} **\dir{} ; (5,0)*{} ; (5,30)*{} **\dir{.} ; (10,0)*{} ; (10,30)*{} **\dir{.} ; (15,0)*{} ; (15,30)*{} **\dir{.} ; (20,0)*{} ; (20,30)*{} **\dir{.} ; (25,0)*{} ; (25,30)*{} **\dir{.} ; (30,0)*{} ; (30,30)*{} **\dir{} ; (5,30)*{} ; (5,20)*{} **\dir{-} ;  (5,20)*{} ;  (10,20)*{} **\dir{-} ; (10,20)*{} ; (15,20)*{} **\dir{-} ; (15,20)*{} ; (15,10)*{} **\dir{-} ; (15,10)*{} ; (20,10)*{} **\dir{-} ; (20,10)*{} ; (20,5)*{} **\dir{-} ; (20,5)*{} ; (30,5)*{} **\dir{-} ;  (-5,15)*{\hat{F}_3} ; (5,5)*{\bullet} ; (0,0)*{(A_3,A_4)} ;  (25,25)*{\mathbb{C}} ;                                                                                                        ; (50,0)*{} ; (80,0)*{} **\dir{} ; (50,5)*{} ; (80,5)*{} **\dir{.} ; (50,10)*{} ; (80,10)*{} **\dir{.} ; (50,15)*{} ; (80,15)*{} **\dir{.} ; (50,20)*{} ; (80,20)*{} **\dir{.} ; (50,25)*{} ; (80,25)*{} **\dir{.} ; (50,30)*{} ; (80,30)*{} **\dir{} ;                   (50,0)*{} ; (50,30)*{} **\dir{} ; (55,0)*{} ; (55,30)*{} **\dir{.} ; (60,0)*{} ; (60,30)*{} **\dir{.} ; (65,0)*{} ; (65,30)*{} **\dir{.} ; (70,0)*{} ; (70,30)*{} **\dir{.} ; (75,0)*{} ; (75,30)*{} **\dir{.} ; (80,0)*{} ; (80,30)*{} **\dir{} ;  (55,30)*{} ; (55,15)*{} **\dir{-} ;  (55,15)*{} ; (65,15)*{} **\dir{-} ; (65,15)*{} ; (65,10)*{} **\dir{-} ; (65,10)*{} ; (75,10)*{} **\dir{-} ; (75,10)*{} ; (75,5)*{} **\dir{-} ; (75,5)*{} ; (80,5)*{} **\dir{-} ;  (45,15)*{\hat{F}_4} ; (55,5)*{\bullet} ; (50,0)*{(A_4,A_1)} ;  (75,25)*{\mathbb{C}} ;                                                                                                                                                                                                                                                                                                                                                              
	\endxy 
	\end{displaymath}
\end{example}

\begin{example}
	Let $\mF$ be a locally free toric sheaf of rank $1$ on the Hirzebruch orbifold $\mH_r^{ab}$. The charts $\mU_1$ and $\mU_4$ has a box of size $a$, while the charts $\mU_2$ and $\mU_3$ has a box of size $b$. Since the rank is $1$, the only possible choice for nonzero $_{b_i} \hat{F}_i$ is 
	$$b_1=(k/a,0), b_2=(0,j/b), b_3=(j/b,0), b_4=(0,k/a).$$
	
	For fixed $0 \leq j < b$ and $0 \leq k < a$, The $\mT$-weights of the generator on each chart are given by
	$$\begin{array}{ll}
	\mU_1: (k/a+A_1)(a,0)+ A_2(0,1),   & 
	\mU_2:  A_2(r,1) + (j/b+A_3)(-b,0),       \\
	\mU_3: (j/b+A_3) (-b,0) + A_4 (-r,-1),       & 
	\mU_4:  A_4(0,-1) +(k/a+A_1)(a,0).      \\   	
	\end{array}
	$$
	Set
	$$B_1=k+aA_1, B_2=A_2, B_3=j+bA_3, B_4=A_4.$$
	The sheaf $\mF$ is uniquely determined by $B_i$. We will show below that the fine grading is also determined.
	
	Suppose the $\mu_b$-weight of the generator is $m_1$ on chart $\mU_1$, then 
	$$_{(k/a, 0)} F_1(A_1,A_2)_{m_1} = _{(k/a, 0)} G_1(0,A_2)_{m_1- a A_1} = _{(k/a, 0)} F_1(\infty,A_2)_{m_1-a A_1}.$$
	The first equation of the gluing conditions implies that 
	$$m_1 \equiv k+ a A_1 +j- r A_2  \equiv B_1+B_3-r B_2    \mod b.$$
	
	Similarly, one can show that the fine gradings of all the generators are determined as follows:
	$$\begin{array}{ll}
	B_1 + B_3 - r B_2  \mod b \text{ on } \mU_1, &
	B_1 + B_3 - r B_2  \mod a \text{ on } \mU_2, \\
	B_1 + B_3 + r B_4  \mod a \text{ on } \mU_3, &
	B_1 + B_3 + r B_4  \mod b \text{ on } \mU_4.
	\end{array}$$
\end{example}

Denote by $L_{(B_1,B_2,B_3,B_4)}$ the $\mT$-equivariant locally free sheaf of rank $1$ corresponding to $(B_1,B_2,B_3,B_4) \in \mZ^4$.  

\begin{proposition} 
	Let $\normalfont{\text{Pic}}^T(\mH_r^{ab})$ be the $\mT$-equivariant Picard group of the Hirzebruch orbifold. Then 
	$$ (B_1,B_2,B_3,B_4) \in \mZ^4 \longmapsto L_{(B_1,B_2,B_3,B_4)} \in \normalfont{\text{Pic}}^T(\mH_r^{ab})$$ \label{24}
	is a group isomorphism. 
\end{proposition}

\begin{remark} \label{25}
	The non-equivariant Picard group of the Hirzebruch orbifold  $\mH_r^{ab}$ is $\mZ \oplus \mZ$ and
	$$\begin{array}{ll} 
	L_{(1,0,0,0)}=(-1,0), &
	L_{(0,0,1,0)}=(-1,0), \\
	L_{(0,1,0,0)}=(0,-1), &
	L_{(0,0,0,1)}=(-r,-1). 
	\end{array}$$
\end{remark}

\begin{example}
	Let $\mF$ be a  locally free toric sheaf of rank $2$ on the Hirzebruch surface $\mH_r^{11}$. On each chart, $\hat{F}_i$ can be described by a double filtration of $\mC^2$:
	
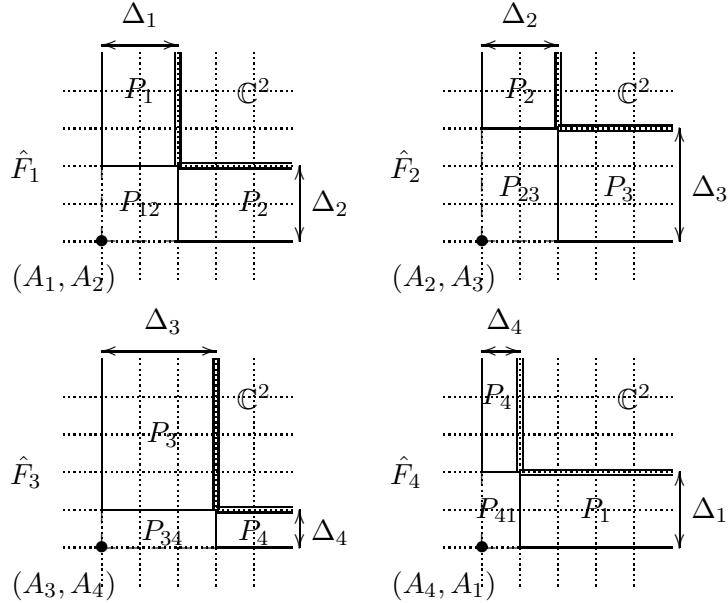
\begin{displaymath}
	\xy
	(0,5)*{} ; (30,5)*{} **\dir{.} ; (0,10)*{} ; (30,10)*{} **\dir{.} ; (0,15)*{} ; (30,15)*{} **\dir{.} ; (0,20)*{} ; (30,20)*{} **\dir{.} ; (0,25)*{} ; (30,25)*{} **\dir{.} ;  (5,0)*{} ; (5,30)*{} **\dir{.} ; (10,0)*{} ; (10,30)*{} **\dir{.} ; (15,0)*{} ; (15,30)*{} **\dir{.} ; (20,0)*{} ; (20,30)*{} **\dir{.} ; (25,0)*{} ; (25,30)*{} **\dir{.} ; (5,30)*{} ; (5,15)*{} **\dir{-} ;  (5,15)*{} ; (15,15)*{} **\dir{-} ; (15,15)*{} ; (15,5)*{} **\dir{-} ; (15,5)*{} ; (30,5)*{} **\dir{-} ; (15,30)*{} ; (15,15)*{} **\dir{=} ; (15,15)*{} ; (30,15)*{} **\dir{=} ;
	(5,15)*{} ; (5,5)*{} **\dir{--} ; (5,5)*{} ; (15,5)*{} **\dir{--} ; (-5,15)*{\hat{F}_1} ; (5,5)*{\bullet} ; (0,0)*{(A_1,A_2)} ;  (25,25)*{\mathbb{C}^2} ; 
	(10,25)*{P_1} ; (25,10)*{P_2} ; (10,10)*{P_{12}} ;  
	(31,5)*{} ;  (31,15)*{} **\dir{-} *\dir{>}; (31,5)*{} *\dir{<}; (35,10)*{\Delta_2}; 
	(5,31)*{} ;  (15,31)*{} **\dir{-} *\dir{>}; (5,31)*{} *\dir{<}; (10,35)*{\Delta_1};  ;                                                                                                        (50,5)*{} ; (80,5)*{} **\dir{.} ; (50,10)*{} ; (80,10)*{} **\dir{.} ; (50,15)*{} ; (80,15)*{} **\dir{.} ; (50,20)*{} ; (80,20)*{} **\dir{.} ; (50,25)*{} ; (80,25)*{} **\dir{.} ; (55,0)*{} ; (55,30)*{} **\dir{.} ; (60,0)*{} ; (60,30)*{} **\dir{.} ; (65,0)*{} ; (65,30)*{} **\dir{.} ; (70,0)*{} ; (70,30)*{} **\dir{.} ; (75,0)*{} ; (75,30)*{} **\dir{.} ;  (55,30)*{} ; (55,20)*{} **\dir{-} ; (55,20)*{} ; (65,20)*{} **\dir{-} ; (65,20)*{} ; (65,5)*{} **\dir{-} ; (65,5)*{} ; (80,5)*{} **\dir{-}; (65,30)*{} ; (65,20)*{} **\dir{=}; (65,20)*{}; (80,20)*{} **\dir{=}; (55,20)*{} ; (55,5)*{} **\dir{--} ; (55,5)*{} ; (65,5)*{} **\dir{--} ; (45,15)*{\hat{F}_2} ; (55,5)*{\bullet} ; (50,0)*{(A_2,A_3)} ;  (75,25)*{\mathbb{C}^2} ;   
	(60,25)*{P_2} ; (73,12)*{P_3} ; (60,12)*{P_{23}};
	(81,5)*{} ;  (81,20)*{} **\dir{-} *\dir{>}; (81,5)*{} *\dir{<}; (85,12)*{\Delta_3};  
	(55,31)*{} ;  (65,31)*{} **\dir{-} *\dir{>}; (55,31)*{} *\dir{<}; (60,35)*{\Delta_2}; 
	;                           
	\endxy 
	\end{displaymath}
	\begin{displaymath}
	\xy
	(0,5)*{} ; (30,5)*{} **\dir{.} ; (0,10)*{} ; (30,10)*{} **\dir{.} ; (0,15)*{} ; (30,15)*{} **\dir{.} ; (0,20)*{} ; (30,20)*{} **\dir{.} ; (0,25)*{} ; (30,25)*{} **\dir{.} ; (5,0)*{} ; (5,30)*{} **\dir{.} ; (10,0)*{} ; (10,30)*{} **\dir{.} ; (15,0)*{} ; (15,30)*{} **\dir{.} ; (20,0)*{} ; (20,30)*{} **\dir{.} ; (25,0)*{} ; (25,30)*{} **\dir{.} ;  (5,30)*{} ; (5,10)*{} **\dir{-} ;  (5,10)*{} ; (20,10)*{} **\dir{-} ; (20,10)*{} ; (20,5)*{} **\dir{-} ; (20,5)*{} ; (30,5)*{} **\dir{-} ;
	(20,30)*{} ; (20,10)*{} **\dir{=};(20,10)*{} ; (30,10)*{} **\dir{=} ; (5,10)*{} ; (5,5)*{} **\dir{--} ; (5,5)*{} ; (20,5)*{} **\dir{--} ;  
	(-5,15)*{\hat{F}_3} ; (5,5)*{\bullet} ; (0,0)*{(A_3,A_4)} ;  (25,25)*{\mathbb{C}^2} ;
	(13,20)*{P_3} ; (25,7)*{P_4} ; (13,7)*{P_{34}} ;  
	(31,5)*{} ;  (31,10)*{} **\dir{-} *\dir{>}; (31,5)*{} *\dir{<}; (35,7)*{\Delta_4}; 
	(5,31)*{} ;  (20,31)*{} **\dir{-} *\dir{>}; (5,31)*{} *\dir{<}; (13,35)*{\Delta_3};  ;                                                                                                        (50,5)*{} ; (80,5)*{} **\dir{.} ; (50,10)*{} ; (80,10)*{} **\dir{.} ; (50,15)*{} ; (80,15)*{} **\dir{.} ; (50,20)*{} ; (80,20)*{} **\dir{.} ; (50,25)*{} ; (80,25)*{} **\dir{.} ; (55,0)*{} ; (55,30)*{} **\dir{.} ; (60,0)*{} ; (60,30)*{} **\dir{.} ; (65,0)*{} ; (65,30)*{} **\dir{.} ; (70,0)*{} ; (70,30)*{} **\dir{.} ; (75,0)*{} ; (75,30)*{} **\dir{.} ;   (55,30)*{} ; (55,15)*{} **\dir{-} ; (55,15)*{} ; (60,15)*{} **\dir{-} ; (60,15)*{} ; (60,5)*{} **\dir{-} ; (60,5)*{} ; (80,5)*{} **\dir{-} ; 
	(60,30)*{} ; (60,15)*{} **\dir{=} ; (60,15)*{} ; (80,15)*{} **\dir{=}; (55,15)*{} ; (55,5)*{} **\dir{--} ; (55,5)*{} ; (60,5)*{} **\dir{--} ; (45,15)*{\hat{F}_4} ; (55,5)*{\bullet} ; (50,0)*{(A_4,A_1)} ;  (75,25)*{\mathbb{C}^2};
	(57,25)*{P_4} ; (70,10)*{P_1} ; (57,10)*{P_{41}};
	(81,5)*{} ;  (81,15)*{} **\dir{-} *\dir{>}; (81,5)*{} *\dir{<}; (85,10)*{\Delta_1};  
	(55,31)*{} ;  (60,31)*{} **\dir{-} *\dir{>}; (55,31)*{} *\dir{<}; (58,35)*{\Delta_4};  ;                                                                                                                                                                                                                                                                                                                                                              
	\endxy 
	\end{displaymath}
	
	Hence $\mF$ is fully determined by $A_1, A_2, A_3, A_4 \in \mZ$, $\Delta_1, \Delta_2, \Delta_3, \Delta_4 \in \mZ_{\geq 0}$ and $P_1, P_2, P_3, P_4 \subset \mC^2$, which can also be viewed as a point $(P_1,P_2,P_3,P_4) \in (\mP^1)^4$. The label $P_{ij}$ stands for the vector space $P_i \cap P_j$.
	
	Generally, for torsion free toric sheaves, the double filtrations may not have strict corners \cite{Koo10}.
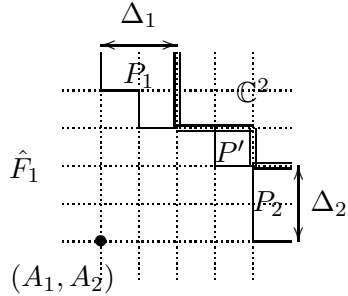
	\begin{displaymath}
	\xy
	(0,5)*{} ; (30,5)*{} **\dir{.} ; (0,10)*{} ; (30,10)*{} **\dir{.} ; (0,15)*{} ; (30,15)*{} **\dir{.} ; (0,20)*{} ; (30,20)*{} **\dir{.} ; (0,25)*{} ; (30,25)*{} **\dir{.} ;  (5,0)*{} ; (5,30)*{} **\dir{.} ; (10,0)*{} ; (10,30)*{} **\dir{.} ; (15,0)*{} ; (15,30)*{} **\dir{.} ; (20,0)*{} ; (20,30)*{} **\dir{.} ; (25,0)*{} ; (25,30)*{} **\dir{.} ; (5,30)*{} ; (5,25)*{} **\dir{-} ;  (5,25)*{} ; (10,25)*{} **\dir{-} ; (10,25)*{}; (10,20)*{} **\dir{-} ; (10,20)*{} ; (15,20)*{} **\dir{-}; (20,20)*{}; (20,15)*{} **\dir{-} ; (20,15)*{} ; (25,15)*{} **\dir{-}; (25,15)*{} ; (25,5)*{} **\dir{-}; (25,5)*{} ; (30,5)*{} **\dir{-};   (15,30)*{}; (15,20)*{} **\dir{=} ; (15,20)*{} ; (25,20)*{} **\dir{=}; (25,20)*{}; (25,15)*{} **\dir{=} ; (25,15)*{} ; (30,15)*{} **\dir{=};
	(-5,15)*{\hat{F}_1} ; (5,5)*{\bullet} ; (0,0)*{(A_1,A_2)} ;  (25,25)*{\mathbb{C}^2} ; 
	(10,27)*{P_1} ; (27,10)*{P_2} ; (22,17)*{P'} ;   
	(31,5)*{} ;  (31,15)*{} **\dir{-} *\dir{>}; (31,5)*{} *\dir{<}; (35,10)*{\Delta_2}; 
	(5,31)*{} ;  (15,31)*{} **\dir{-} *\dir{>}; (5,31)*{} *\dir{<}; (10,35)*{\Delta_1};  ;                                                                                                                        
	\endxy 
	\end{displaymath}
	
\end{example}

\begin{example} \label{27}
	Let $\mF$ be a locally free toric sheaf of rank $2$ on the Hirzebruch orbifold $\mH_r^{ab}$. Since the rank is $2$, either $1$ or $2$ box summands are nonempty. There are $4$ possible choices for $ _{b_i} \hat{F}_i$ to be nonzero. 
	
	\begin{enumerate}
		\item $ _{b_i} \hat{F}_i \neq 0$  for $b_1=(k/a,0), b_2=(0,j/b), b_3=(j/b,0), b_4=(0,k/a). $ 
		
		On each chart, it is described by a double filtration as for $\mH_{r}^{11}$.
		
		Since we will later work on stable toric sheaves and the decomposable sheaves are not stable, we would like to classify all the types of indecomposable toric sheaves. They are listed below: 
		\begin{enumerate} 
			\item $\Delta_i > 0$ for all $i$. $P_i$'s are mutually distinct.
			\item $\Delta_{i'} = 0$ for some unique $i'$ and $\Delta_i > 0$ for $i \neq i'$. $P_{i'}$ is omitted and $P_i$'s are mutually distinct for $i \neq {i'}$. 
			\item $\Delta_i > 0$ for all $i$. Only two of $P_i$'s are same.
		\end{enumerate}
		
		\item $_{b_i} \hat{F}_i \neq 0$ for $b_1=(k/a,0), b_2=(0,j/b), b_2=(0,j'/b),
		b_3=(j/b,0), b_3=(j'/b,0), b_4=(0,k/a).$
		
		Suppose $A_2' - A_2 = \Delta_2 \geq 0$ and  $A_4'-A_4 = \Delta_4 \geq 0$. Denote $\Delta_3= A_3'-A_3$. Sheaves of this type are fully determined by $A_1, A_2, A_3, A_4, b \nmid \Delta_3 \in \mZ$, $\Delta_1, \Delta_2, \Delta_4 \in \mZ_{\geq 0}$, and $P_1 \neq P_2 \subset \mC^2$. They are decomposable and can be described as follows: 
		
		\begin{displaymath}
		\xy
		(0,5)*{} ; (30,5)*{} **\dir{.} ; (0,10)*{} ; (30,10)*{} **\dir{.} ; (0,15)*{} ; (30,15)*{} **\dir{.} ; (0,20)*{} ; (30,20)*{} **\dir{.} ; (0,25)*{} ; (30,25)*{} **\dir{.} ;  (5,0)*{} ; (5,30)*{} **\dir{.} ; (10,0)*{} ; (10,30)*{} **\dir{.} ; (15,0)*{} ; (15,30)*{} **\dir{.} ; (20,0)*{} ; (20,30)*{} **\dir{.} ; (25,0)*{} ; (25,30)*{} **\dir{.} ; (5,30)*{} ; (5,15)*{} **\dir{-} ;  (5,15)*{} ; (15,15)*{} **\dir{-} ; (15,15)*{} ; (15,5)*{} **\dir{-} ; (15,5)*{} ; (30,5)*{} **\dir{-} ; (15,30)*{} ; (15,15)*{} **\dir{=} ; (15,15)*{} ; (30,15)*{} **\dir{=} ;
		(-5,15)*{_{(k/a,0)}\hat{F}_1} ; (5,5)*{\bullet} ; (0,0)*{(A_1,A_2)} ;  (25,25)*{\mathbb{C}^2} ; 
		(10,25)*{P_2} ; (25,10)*{P_1} ;  
		(31,5)*{} ;  (31,15)*{} **\dir{-} *\dir{>}; (31,5)*{} *\dir{<}; (35,10)*{\Delta_2}; 
		(5,31)*{} ;  (15,31)*{} **\dir{-} *\dir{>}; (5,31)*{} *\dir{<}; (10,35)*{\Delta_1};  
		;            
		(55,25)*{} ; (75,25)*{} **\dir{.}; (55,30)*{} ; (75,30)*{} **\dir{.} ; (55,35)*{} ; (75,35)*{} **\dir{.}; (60,20)*{} ; (60,40)*{} **\dir{.} ; (65,20)*{} ; (65,40)*{} **\dir{.}; (70,20)*{} ; (70,40)*{} **\dir{.};  (60,40)*{} ; (60,25)*{} **\dir{-};(60,25)*{} ; (75,25)*{} **\dir{-} ;  (45,30)*{_{(0,j/b)}\hat{F}_2} ; (60,25)*{\bullet} ; (55,20)*{(A_2 ,A_3)} ;  (70,35)*{P_1} ; 
		;
		(55,-5)*{} ; (75,-5)*{} **\dir{.}; (55,0)*{} ; (75,0)*{} **\dir{.} ; (55,5)*{} ; (75,5)*{} **\dir{.}; (60,-10)*{} ; (60,10)*{} **\dir{.} ; (65,-10)*{} ; (65,10)*{} **\dir{.}; (70,-10)*{} ; (70,10)*{} **\dir{.};  (60,10)*{} ; (60,-5)*{} **\dir{-};(60,-5)*{} ; (75,-5)*{} **\dir{-} ;  (45,0)*{_{(0,j'/b)}\hat{F}_2} ; (60,-5)*{\bullet} ; (55,-10)*{(A_2' ,A_3')} ;  (70,5)*{P_2} ; 
		\endxy 
		\end{displaymath}		
		\begin{displaymath}
		\xy
		(5,25)*{} ; (25,25)*{} **\dir{.}; (5,30)*{} ; (25,30)*{} **\dir{.} ; (5,35)*{} ; (25,35)*{} **\dir{.}; (10,20)*{} ; (10,40)*{} **\dir{.} ; (15,20)*{} ; (15,40)*{} **\dir{.}; (20,20)*{} ; (20,40)*{} **\dir{.};  (10,40)*{} ; (10,25)*{} **\dir{-};(10,25)*{} ; (25,25)*{} **\dir{-} ;  (-5,30)*{_{(j/b,0)}\hat{F}_3} ; (10,25)*{\bullet} ; (5,20)*{(A_3 ,A_4)} ;  (20,35)*{P_1} ; 
		;
		(5,-5)*{} ; (25,-5)*{} **\dir{.}; (5,0)*{} ; (25,0)*{} **\dir{.} ; (5,5)*{} ; (25,5)*{} **\dir{.}; (10,-10)*{} ; (10,10)*{} **\dir{.} ; (15,-10)*{} ; (15,10)*{} **\dir{.}; (20,-10)*{} ; (20,10)*{} **\dir{.};  (10,10)*{} ; (10,-5)*{} **\dir{-};(10,-5)*{} ; (25,-5)*{} **\dir{-} ;  (-5,0)*{_{(j'/b,0)}\hat{F}_3} ; (10,-5)*{\bullet} ; (5,-10)*{(A_3' ,A_4')} ;  (20,5)*{P_2} ;  
		;  
		(50,5)*{} ; (80,5)*{} **\dir{.} ; (50,10)*{} ; (80,10)*{} **\dir{.} ; (50,15)*{} ; (80,15)*{} **\dir{.} ; (50,20)*{} ; (80,20)*{} **\dir{.} ; (50,25)*{} ; (80,25)*{} **\dir{.} ; (55,0)*{} ; (55,30)*{} **\dir{.} ; (60,0)*{} ; (60,30)*{} **\dir{.} ; (65,0)*{} ; (65,30)*{} **\dir{.} ; (70,0)*{} ; (70,30)*{} **\dir{.} ; (75,0)*{} ; (75,30)*{} **\dir{.} ;  (55,30)*{} ; (55,15)*{} **\dir{-} ; (55,15)*{} ; (65,15)*{} **\dir{-} ; (65,15)*{} ; (65,5)*{} **\dir{-} ; (65,5)*{} ; (80,5)*{} **\dir{-}; (65,30)*{} ; (65,15)*{} **\dir{=}; (65,15)*{}; (80,15)*{} **\dir{=}; (55,15)*{} ;  
		(45,15)*{_{(0,k/a)}\hat{F}_4} ; (55,5)*{\bullet} ; (50,0)*{(A_4,A_1)} ;  (75,25)*{\mathbb{C}^2} ;   
		(60,25)*{P_1} ; (75,10)*{P_2} ; 
		(81,5)*{} ;  (81,15)*{} **\dir{-} *\dir{>}; (81,5)*{} *\dir{<}; (85,10)*{\Delta_1};  
		(55,31)*{} ;  (65,31)*{} **\dir{-} *\dir{>}; (55,31)*{} *\dir{<}; (60,35)*{\Delta_4}; 
		;                           
		\endxy 
		\end{displaymath}
		
		\item $_{b_i} \hat{F}_i \neq 0$ for $b_1=(k/a,0), b_1=(k'/a,0), b_2=(0,j/b), b_3=(j/b,0), b_4=(0,k/a), b_4=(0,k'/a).$
		
		It's similar to the second case and all the toric sheaves of this type are decomposable.
		
		\item Two box summands are nonzero for all the charts. 
		
		It can be easily seen that $\mF$ is decomposable in this case.
		
	\end{enumerate}
\end{example}  

\section{Hilbert Polynomial}

\subsection{K-Group} ~\\

Let $K_0(\mH_r^{ab})$ be the Grothendieck group of coherent sheaves on the Hirzebruch orbifold $\mH_r^{ab}$. By \cite{BH06}, $K_0(\mH_r^{ab})_\mQ$ is isomorphic to the quotient ring
$\mQ[g^{\pm}, h^{\pm}]/I$ where $I$ is generated by 
$$\left\{ \begin{array}{l}
(1-g^a)(1-g^b)(1-h) \\
(1-g^a)(1-g^b)(1-g^rh)\\
(1-g^a)(1-h)(1-g^rh)\\
(1-g^b)(1-h)(1-g^rh).
\end{array}\right.$$
Here $g:=[(-1,0)], h:=[(0,-1)]$ are $K$-group classes of the generators of $\text{Pic}(\mH_r^{ab}) \cong \mZ \oplus \mZ$.

Recall that the $\mT$-action on $\mH_r^{ab}$ has four fixed points corresponding to the origin of each chart. Denote them by $P_1, P_2, P_3, P_4$.

\begin{proposition} \label{31}
	In $K_0(\mH_r^{ab})$, we have
	$$\begin{array}{ll}
	\left[ \mO_{P_1} \otimes \hat{\mu}_b^i \right] = (1-g^a)(1-h)g^i,  &        
	\left[ \mO_{P_2} \otimes \hat{\mu}_a^i \right]= (1-g^b)(1-h)g^i,  \\
	\left[ \mO_{P_3} \otimes \hat{\mu}_a^i \right] = (1-g^b)(1-g^rh)g^i,  &
	\left[ \mO_{P_4} \otimes \hat{\mu}_b^i \right] = (1-g^a)(1-g^rh)g^i.  
	\end{array}$$
\end{proposition}

\begin{proof}
	The sheaf $\mO_{P_1} \otimes \hat{\mu}_b^i$ is described by a $S$-family where $\hat{F}_2=\hat{F}_3=\hat{F}_4=0$ and $\hat{F}_1$ only consists of a single vector space $\mC$ with $\mu_b$-weight $i$ at the position $(0,0)$.
	
	\begin{displaymath}
	\xy
	(0,0)*{} ; (30,0)*{} **\dir{} ; (0,5)*{} ; (30,5)*{} **\dir{.} ; (0,10)*{} ; (30,10)*{} **\dir{.} ; (0,15)*{} ; (30,15)*{} **\dir{.} ; (0,20)*{} ; (30,20)*{} **\dir{.} ; (0,25)*{} ; (30,25)*{} **\dir{.} ; (0,30)*{} ; (30,30)*{} **\dir{} ;                   (0,0)*{} ; (0,30)*{} **\dir{} ; (5,0)*{} ; (5,30)*{} **\dir{.} ; (10,0)*{} ; (10,30)*{} **\dir{.} ; (15,0)*{} ; (15,30)*{} **\dir{.} ; (20,0)*{} ; (20,30)*{} **\dir{.} ; (25,0)*{} ; (25,30)*{} **\dir{.} ; (30,0)*{} ; (30,30)*{} **\dir{} ; (5,30)*{} ; (5,5)*{} **\dir{-} ; (5,5)*{} ; (30,5)*{} **\dir{-} ;  (10,30)*{} ; (10,5)*{} **\dir{--} ; (5,10)*{} ; (30,10)*{} **\dir{--} ; (-5,20)*{\hat{F}_1} ; (5,5)*{\bullet} ; (0,2.5)*{(0,0)} ; (15,12.5)*{(1,1)} ; (15,2.5)*{(1,0)} ; (0,12.5)*{(0,1)};  (25,25)*{\mathbb{C}} ;                                                                                                                             
	\endxy 
	\end{displaymath}
	
	Using the description of the line bundle introduced in \hyperref[24]{Proposition \ref*{24}}, we can construct the exact sequence:
	$$0 \longrightarrow L_{(a \cdot 1,1,0,0)} \longrightarrow   L_{(a \cdot 1,0,0,0)} \oplus L_{(0,1,0,0)} \longrightarrow L_{(0,0,0,0)}  \longrightarrow \mO_{P_1} \longrightarrow 0 . $$
	Hence 
	$$[\mO_{P_1}]=1+g^ah-g^a-h=(1-g^a)(1-h).$$
	Since $B_1=a A_1=0, B_2=A_2=0$, the fine grading of $\mO_{P_1} \otimes \hat{\mu}_b^i$ is equal to $B_3 \text{ mod } b$ on $\mU_1$. As a result, 
	$$[ \mO_{P_1} \otimes \hat{\mu}_b^i ] = [ \mO_{P_1} \otimes L_{(0,0,i,0)}] = (1-g^a)(1-h)g^i. $$
	The calculation for other charts is similar.
\end{proof}

Now let's consider the general case. Suppose there is a $S$-family such that $\hat{F}_2=\hat{F}_3=\hat{F}_4=0$ and $\hat{F}_1$ consists of a single space $\mC$ with $\mu_b$-grading $i$ at the position $(k/a+A_1)(a,0) + A_2(0,1).$
Then the corresponding sheaf is $\mO_{P_1} \otimes L_{(k+ aA_1 ,A_2, i-k-a A_1 + r A_2,0)}$ and its class in $K_0(\mH_r^{ab})$ is 
$$(1-g^a)(1-h)g^{i+rA_2}h^{A_2} =(1-g^a)(1-h)g^i.$$

As a result, the class of such a sheaf in $K_0(\mH_r^{ab})$ only depends on the fine grading. This is quite useful when we calculate the Hilbert polynomial later.

\subsection{Riemann-Roch} ~\\

Riemann-Roch on Deligne-Mumford stacks was first proved in \cite{Toe99}. Later, \cite{Edi12} gave a simpler proof based on the equivariant localization theorem. In our paper, we will follow the notation of  inertia stacks used in the appendix of \cite{Tse10}, which is essentially same as \cite[Section 4]{Edi12}.  

Recall from \cite{BCS05} that for each $d$-dimensional cone in the fan $\Sigma$, Box($\sigma$) is the set of	elements $v \in N \cong \mZ^2$ such that $v=\sum_{\rho_i \in \sigma} q_i b_i$ where $b_i$ is the $i$th column of the matrix $B$ \eqref{1.3.3} and $0 \leq q_i < 1$ with $q_i \in \mQ$. Denote by Box($\Sigma$) the union of Box($\sigma$) for all $d$-dimensional cones. 

Since $\mH_r^{ab} \cong [Z/G]$ is a quotient stack, each component of its inertia stack is isomorphic to $[Z^g/G]$ where $Z^g$ denotes the locus of points fixed in $Z$ by $g$. By \cite{BCS05}, the elements $v \in \text{Box}(\Sigma)$ are in one-to one correspondence with elements $g \in G$ that fix a point of $Z$. 

Suppose $\text{gcd}(a,b)=1$. A box element for the stacky fan of the Hirzebruch orbifold $\mH_r^{ab}$ can be in Box($\rho_1$), Box($\rho_3$) or Box($\sigma_i$) for $1 \leq i \leq 4$. Hence to find all the components of the inertia stack, we classify all the substacks which correspond to the minimal cones that contain the box elements.

If a box element is on Box($\rho_1$), then $x=0$ and the corresponding stabilizer $g=(\tau,\lambda)$ must satisfy  $\lambda=1, \tau^b=1, \tau^r \lambda=1$. Suppose $\text{gcd}(r,b)=p$, then 
$$g=(e^{2\pi \sqrt{-1} \frac{l}{p}}, 1), l=1,...,p-1.$$	
Hence the corresponding component of the inertia stack is 
$$ \mX_{\rho_1} \cong [Z^g/G] \cong
[\mC^3-V(yz,zw)/(\mC^*)^2],  \quad (\tau,\lambda): (y,z,w) \to (\lambda y, \tau^b z, \tau^r \lambda w).$$

Let $\text{gcd}(r,a)=q$. We summarize all the components in the table below:

$$\begin{tabular}{|c|c|c|c|}
\hline 
& stablizer $g$ & substack $[Z^g/G]$ &  $(\tau,\lambda) \in (\mC^*)^2$-action \\[0.2em] \hline 
$\rho_1$  &   $\begin{array}{c}
(e^{2\pi \sqrt{-1} \frac{l}{p}}, 1) \\
l=1,...,p-1
\end{array}$    
&   $[\mC^3-V(yz,zw)/(\mC^*)^2]$   &   $(y,z,w)   \to   (\lambda y, \tau^b z, \tau^r \lambda w)$   \\[0.2em] \hline
$\rho_3$  &   $\begin{array}{c}
(e^{2\pi \sqrt{-1} \frac{l}{q}}, 1)\\ l=1,...,q-1
\end{array}$   
&   $[\mC^3-V(xy,wx)/(\mC^*)^2]  $ &   $(x,y,w)   \to   (\tau^a x, \lambda y, \tau^r \lambda w)$   \\[0.2em] \hline
$\sigma_1$ &  $\begin{array}{c} (e^{2\pi \sqrt{-1} \frac{l}{b}}, e^{-2\pi \sqrt{-1} \frac{sal}{b}} ) \\ \frac{b}{p} \nmid l, l=1,...,b-1 
\end{array} $   
&   $[\mC^2-V(zw)/(\mC^*)^2]$   &   $(z,w)   \to   (\tau^b z, \tau^r \lambda w)$   \\[0.2em] \hline
$\sigma_2$   &   $\begin{array}{c}
(e^{2\pi \sqrt{-1} \frac{l}{a}}, e^{-2\pi \sqrt{-1} \frac{tbl}{a}} ) \\ \frac{a}{q} \nmid l, l=1,...,a-1 
\end{array}$  
&   $[\mC^2-V(xw)/(\mC^*)^2]$   &   $(x,w)   \to   (\tau^a x, \tau^r \lambda w)$   \\[0.2em] \hline
$\sigma_3$   &   $\begin{array}{c} (e^{2\pi \sqrt{-1} \frac{l}{a}}, 1 ) \\  \frac{a}{q} \nmid l, l=1,...,a-1
\end{array}$  
&    $[\mC^2-V(xy)/(\mC^*)^2]$   &   $ (x,y)   \to   (\tau^a x, \lambda y)$   \\[0.2em] \hline
$\sigma_4$   & $\begin{array}{c} (e^{2\pi \sqrt{-1} \frac{l}{b}}, 1 ) \\
\frac{b}{p} \nmid l, l=1,...,b-1 
\end{array}$ 
&   $[\mC^2-V(yz)/(\mC^*)^2]$   &  $ (y,z)   \to   ( \lambda y, \tau^b z)$ \\[0.2em] \hline
\end{tabular}
$$

Write $I\mH_r^{ab}$ for the inertia stack of the Hirzebruch orbifold $\mH_r^{ab}$. Let $\pi:I\mH_r^{ab} \rightarrow \mH_r^{ab}$ be the natural projection.
Suppose a vector bundle $V$ on $I\mH_r^{ab} $ is decomposed into a direct sum $\oplus_{\zeta_i} V_i$ of eigenbundles with eigenvalue $\zeta_i$. Let $\mu_\infty$ be the group of all roots of unity, then we define $\rho(V):=\sum_{\zeta_i} \zeta_i V^i$
and $\widetilde{ch} :K^0 (\mH_r^{ab} ) \to A^*(I \mH_r^{ab}) \otimes \mu_\infty$ as the composition $$K^0(\mH_r^{ab}) \xrightarrow{\pi^*} K^0(I\mH_r^{ab}) \xrightarrow{\rho} K^0(I\mH_r^{ab}) \otimes \mu_\infty \xrightarrow{ch} A^*(I\mH_r^{ab}) \otimes \mu_\infty.$$
For a line bundle $L$ on $\mH_r^{ab}$, define $\widetilde{Td}:\text{Pic}(\mH_r^{ab}) \rightarrow A^*(I\mH_r^{ab}) \otimes \mu_\infty $ as
$$\widetilde{Td}(L)=\left\{  
\begin{array}{ll}
Td(\pi^*L)    & \text{if the eigenvalue of } \pi^*L \text{ is } 1 \\
\ds \frac{1}{ch(1- \zeta^{-1} \cdot \pi^* L^\vee)} & \text{if the eigenvalue of } \pi^*L \text{ is } \zeta \neq 1.
\end{array} 
\right.
$$
Then by Riemann-Roch, the Euler characteristic of a coherent sheaf $\mF$ on $\mH_r^{ab}$ is given by 
$$\chi(\mF)= \int_{I \mH_r^{ab} } \widetilde{ch} (\mF) \cdot \widetilde{Td}(\mO(\mD_{\rho_1})) \cdot \widetilde{Td}(\mO(\mD_{\rho_2})) \cdot \widetilde{Td}(\mO(\mD_{\rho_3})) \cdot \widetilde{Td}(\mO(\mD_{\rho_4}))$$
where $\mD_{\rho_i}$ is the divisor corresponding to the ray $\rho_i$ in  \hyperref[fig2]{Figure \ref*{fig2}}.

\begin{proposition}
	Suppose $\normalfont{\text{gcd}}(a,b)=1$. Consider the line bundle $(m,n) \in \normalfont{\text{Pic}} (\mH_r^{ab})$ (\hyperref[25]{Remark \ref*{25}}). The Euler characteristic is given as follows:
	\begin{align*}
	\chi( (m,n)) & =  \ds  \frac{1+n}{2a} + \frac{1+n}{2b} + \frac{(1+n)m}{ab} - \frac{n(n+1)r}{2ab} + \sum_{l=1}^{p-1}  \frac{\omega_p^{ml}}{1-\omega_p^{-al}} \frac{n+1}{b}   \\
	& \quad + \ds \sum_{l=1}^{q-1}  \frac{\omega_q^{ml}}{1-\omega_q^{-bl}} \frac{n+1}{a} 
	+ \sum_{\tiny{\begin{array}{c} l=1 \\ \frac{b}{p} \nmid l \end{array}}}^{b-1} \frac{\omega_b^{ml}}{1-\omega_b^{-al}} \left(\frac{{1-\omega_b^{-(n+1)sal} }}{1-\omega_b^{-sal}} \right) \frac{1}{b}  \\
	& \quad + \ds \sum_{\tiny{\begin{array}{c} l=1 \\ \frac{a}{q} \nmid l \end{array}}}^{a-1} \frac{\omega_a^{ml}}{1-\omega_a^{-bl}} \left(\frac{{1-\omega_a^{-(n+1)tbl} }}{1-\omega_a^{-tbl}} \right) \frac{1}{a}.
	\end{align*} 
	where $p=\normalfont{\text{gcd}(b,r)}$, $q=\normalfont{\text{gcd}(a,r)}$ and $\omega_k=e^{\frac{2\pi \sqrt{-1}}{k}}$ for $k=a,b,p,q$. Especially, $\chi(\mO_{\mH_r^{ab}})=\chi(\mO(\mD_{\rho_1}))=\chi(\mO(\mD_{\rho_3}))=1$. 
\end{proposition}

\begin{proof}
	The only $2$-dimensional component of $I\mH_r^{ab}$ is $\mH_r^{ab}$ itself. By \cite{EM13} and \cite{CLS11}, the orbifold Chow ring is 
	$$\mQ[x,y,z,w]/(xz,yw,bx-az,sx+y+tz-w) \cong \mQ[x,y]/(x^2,ay^2+rxy)$$ 
	and $ \int_{\mH_r^{ab}} xy= \frac{1}{b}.$ 
	
	The $1$-dimensional components come from $\rho_1$ and $\rho_3$. By \cite{BCS05}, the substack $[Z^g/G]$ for $\rho_1$ is isomorphic to the substack constructed from the quotient stacky fan $\Sigma / \rho_1$ \cite{BCS05}. One can show that $Z(\rho_1) \cong \mC^2- V(y,w)$ and the action of $G(\rho_1) \cong \mC^* \times \mu_b$ on $Z(\rho_1)$ is given by $(\lambda,\zeta)(y,w) = (\lambda y, \lambda \zeta^s w).$
	Hence the Chow ring is $\mQ[y]/(y^2)$ and 
	$\int_{\mX_{\rho_1}} y= \frac{1}{b}.$ 
	
	Similarly, the Chow ring is $\mQ[y]/(y^2)$ for another type of $1$-dimensional components and $\int_{\mX_{\rho_3}} y= \frac{1}{a}.$ 
	
	There are $4$ types of $0$-dimensional components induced by $\sigma_i$. Two of them are isomorphic to $B\mu_b$, and others $B\mu_a$. The Chow ring is $\mQ$ and 
	$ \int_{B\mu_b} 1= \frac{1}{b}, \int_{B\mu_a} 1= \frac{1}{a}.$
	
	Thus $I\mH_r^{ab}$ is the disjoint union of $7$ types of components in general. Depending on the relations among $a,b$ and $r$, there may be fewer types.
	
	On each type of components, the Chern character of a line bundle $\widetilde{ch}((m,n))$ is given by
	\begin{multline*}
	\left( 1+(\frac{m}{a}x+ny)+\frac{1}{2}(\frac{m}{a}x+ny)^2,(1+ny)\omega_p^{ml}, (1+ny)\omega_q^{ml}, \omega_b^{(m-nsa)l}, \omega_a^{(m-ntb)l}, \omega_a^{ml}, \omega_b^{ml} \right).
	\end{multline*}
	Note that $l$ runs over $\{1,...,p-1\}$ for the $2$nd type, $\{1,...,q-1\}$ for the $3$rd type, $\{1,...,b-1 ; \frac{b}{p} \nmid l \}$ for the $4$th and $7$th types, $\{1,...,a-1; \frac{a}{q} \nmid l\}$ for the $5$th and $6$th types.
	
	One can also show that $\widetilde{Td}(\mO(\mD_{\rho_1})) \cdot \widetilde{Td}(\mO(\mD_{\rho_2})) \cdot \widetilde{Td}(\mO(\mD_{\rho_3})) \cdot \widetilde{Td}(\mO(\mD_{\rho_4}) )$ on each type of components is
	\begin{multline*}
	\left(  1+y+(\frac{b}{2a}+\frac{r}{2a}+\frac{1}{2})x + (\frac{b}{2a}+\frac{1}{2})xy,\frac{1+y}{1-\omega_p^{-al}}, \frac{1+y}{1-\omega_q^{-bl}}, \right. \\
	\left.
	 \frac{1}{(1-\omega_b^{-al})(1-\omega_b^{sal})}  ,   \frac{1}{(1-\omega_a^{-bl})(1-\omega_a^{tbl})} ,
	\frac{1}{(1-\omega_a^{-bl})(1-\omega_a^{-tbl})} , 
	\frac{1}{(1-\omega_b^{-al})(1-\omega_b^{-sal})} \right).
	\end{multline*}
	
	Adding all the integrals together, we get the desired result.
	
	To show $\chi(\mO_{\mH_r^{ab}})=\chi(\mO(\mD_{\rho_1}))=\chi(\mO(\mD_{\rho_3}))=1$, we repeatedly use the following two facts: 
	\begin{itemize}
		\item If $a,p$ are coprime, 
		$\ds \sum_{l=1}^{p-1}  \frac{1}{1-\omega_p^{-al}} 
		= \sum_{l=1}^{p-1}  \frac{1}{1-\omega_p^{l}}$.
		\item $\ds \frac{1}{1-\omega_p^{l}} + \frac{1}{1-\omega_p^{-l}} =1$.
	\end{itemize}
\end{proof}

\subsection{Coarse Moduli Space} ~\\

Suppose gcd$(a,b)$=1. The coarse moduli space of the Hirzebruch orbifold $\mH_r^{ab}$ is a toric variety $\mathrm{H}$ given by the following fan
$$
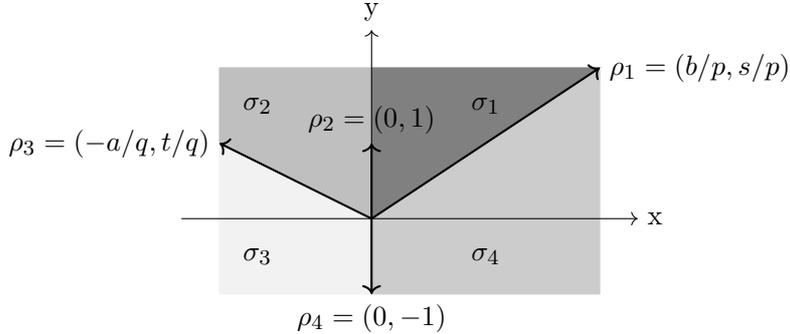
\begin{tikzpicture}
\filldraw[black!25!white] (0,0) -- (0,2) -- (-2,2) -- (-2,1);
\filldraw[black!50!white] (0,0) -- (3,2) -- (0,2);
\filldraw[black!5!white] (0,0) -- (-2,1) -- (-2,-1) -- (0,-1);
\filldraw[black!20!white] (0,0) -- (3,2) -- (3,-1) -- (0,-1);
\draw[thin,->] (-2.5,0) -- (3.5,0) node[anchor=west] {x};
\draw[thin,->] (0,-1) -- (0,2.5) node[anchor=south] {y};
\draw[thick,->] (0,0) -- (3,2);
\draw[thick,->] (0,0) -- (0,1);
\draw[thick,->] (0,0) -- (0,-1);
\draw[thick,->] (0,0) -- (-2,1);
\node[right] at (3,2) {$\rho_1=(b/p,s/p)$};
\node[left] at (-2,1) {$\rho_3=(-a/q,t/q)$};
\node[above] at (0,1) {$\rho_2=(0,1)$};
\node[below] at (0,-1) {$\rho_4=(0,-1)$};
\node at (1.5,1.5) {$\sigma_1$};
\node at (-1.5,1.5) {$\sigma_2$};
\node at (-1.5,-0.5) {$\sigma_3$};
\node at (1.5,-0.5) {$\sigma_4$};
\end{tikzpicture}
$$  
where $r=sa+bt, p=\text{gcd}(b,r)$ and $q=\text{gcd}(a,r)$. Since $\text{gcd}(a,b)=1$, $\frac{s}{p}$ and $\frac{t}{q}$ are integers. Let $D_i$ be the divisor corresponding to the ray $\rho_i$. Then
$$\left\{ \begin{array}{l}
\ds \frac{b}{p} D_1 \sim \frac{a}{q} D_3 \\
\ds \frac{s}{p} D_1 + D_2 + \frac{t}{q} D_3 \sim D_4.
\end{array} \right.$$

To find the Picard group, we need to determine when a Weil divisor is Cartier. Suppose $D= t_1 D_1  + t_2 D_2 $ is Cartier. Denote by $n_\rho$ the primitive generator of the ray $\rho$. Then for each $\sigma_i$, there exists $m_{\sigma_i}=(x_i,y_i)$ such that $ \langle m_{\sigma_i}, n_\rho \rangle =-t_\rho$ for all $\rho \in \sigma_i(1)$, where $\sigma_i(1)$ denotes the collection of rays of $\sigma_i$ \cite{CLS11}.  

For $\sigma_1$, it implies 
$$ \left\{ \begin{array}{l}
\ds \frac{b}{p} x_1 +\frac{s}{p} y_1 = -t_1 \\
y_1 = -t_2
\end{array}   \right.
$$  
from which we get $\frac{b}{p} \mid -t_1 + \frac{s}{p} t_2$. By checking each $\sigma_i$, one can show that the conditions for $D$ to be Cartier are $\frac{b}{p} \mid t_1, \frac{ba}{pq} \mid t_2$. Therefore
$$ \text{Pic}(\mathrm{H}) \cong \{ t_1 \frac{b}{p} D_1 + t_2 \frac{ba}{pq} D_2 \vert t_1, t_2 \in \mZ\} \cong \{ t_1 \frac{b}{p} D_1 + t_4 \frac{ba}{pq} D_4 \vert t_1, t_4 \in \mZ \} \cong \mZ^2.$$

The Cartier divisor $t_1 \frac{b}{p} D_1 + t_4 \frac{ba}{pq} D_4$ is ample if and only if for each $\sigma_i$, there exists $m_{\sigma_i}=(x_i,y_i)$ such that 
$$ \left\{ \begin{array}{l}
\langle m_{\sigma_i}, n_\rho \rangle =-t_\rho \text{ for all } \rho \in \sigma_i(1) \\
\langle m_{\sigma_i}, n_\rho \rangle > -t_\rho \text{ for all } \rho \in \Sigma(1)/\sigma_i(1). 
\end{array}   \right.
$$ 
One can compute that 
$m_{\sigma_1}=(-t_1, 0), m_{\sigma_2}=(0, 0), m_{\sigma_3}=(\frac{bt}{pq} t_4, \frac{ba}{pq} t_4),  m_{\sigma_4}=(-\frac{sa}{pq}t_4-t_1,\frac{ba}{pq}t_4 )$ is a solution. 

We get several inequalities which reduce to $t_1>0$, $t_4>0$. Thus 
$\mO_{\mathrm{H}}(t_1 \frac{b}{p} D_1 + t_4 \frac{ba}{pq} D_4)$ is ample if and only if $t_1, t_4>0$.

Consider the ample line bundle $L=\mO_{\mathrm{H}}(\frac{b}{p} D_1 + \frac{ba}{pq} D_4)$. By the property of the root stack \cite{FMN10}, $\epsilon: \mH_r^{ab} \to \mathrm{H}$
is a morphism with divisor multiplicities $(p,1,q,1)$. Hence 
$$\epsilon^* L = \mO_{\mH_r^{ab}}(b\mD_{\rho_1} + \frac{ba}{pq}\mD_{\rho_4}) \cong (ba(1+\frac{r}{pq}) , \frac{ba}{pq}) \in \text{Pic}(\mH_r^{ab}).$$

For any coherent sheaf $\mF$ on $\mH_r^{ab}$, we can then define the Hilbert polynomial of $\mF$ with respect to $\epsilon^* L$ as
$$P(\mF,T) := \chi (\mF \otimes \epsilon^* L^T ).$$

\begin{proposition}
	Suppose $\normalfont{\text{gcd}}(a,b)=1$. Consider the line bundle $(m,n) \in \normalfont{\text{Pic}}(\mH_r^{ab})$, then
	$$\begin{array}{ll}
	P((m,n),T) & = \ds  \left( \frac{bar}{2p^2q^2} + \frac{ba}{pq}  \right) T^2 + \left( \frac{a+b+2m+r}{2pq} + n+1 + \right. \\
	& + \ds \left. \sum_{l=1}^{p-1} \frac{\omega_p^{ml}}{1-\omega_p^{-al}} \frac{a}{pq} + \sum_{l=1}^{q-1}  \frac{\omega_q^{ml}}{1-\omega_q^{-bl}} \frac{b}{pq} \right) T + \chi ( (m,n) ).
	\end{array}
	$$
\end{proposition}

\begin{proof}
	To calculate $\chi ( (m+ ba(1+\frac{r}{pq}) T,n+  \frac{ba}{pq} T ) )$, we note that $$\ds \omega_b^{\frac{ba}{pq} sa} = \omega_b^{\frac{ba}{pq} (r-tb)}=1,  \omega_a^{\frac{ba}{pq} tb} = \omega_a^{\frac{ba}{pq} (r-sa)}=1.$$
	Then the result follows.
\end{proof}

\subsection{Modified Hilbert Polynomial} ~\\

By \cite{OS03} \cite{Nir08}, A locally free sheaf $\mE$ on Hirzebruch orbifold $\mH_r^{ab}$ is a generating sheaf if for every geometric point $x$ of $\mH_r^{ab}$, the representation $\mE_x$ of the stabilizer group at that point contains every irreducible representation.

One can show that $ \mE = \bigoplus_{k=0}^{ab-1} (-k,0)$ is a generating sheaf, although is not of minimal rank usually. Let $\epsilon: \mH_r^{ab} \to \mathrm{H}$ be the structure morphism. Fix the generating sheaf $\mE$ as above and the ample invertible sheaf $ L=\mO(\frac{b}{p} D_1 + \frac{ba}{pq} D_4)$. We define the modified Hilbert polynomial for a sheaf $\mF$ on the Hirzebruch orbifold $\mH_r^{ab}$ as
$$P_\mE (\mF, T) =\chi (\mH_r^{ab}, \mF \otimes \mE^\vee \otimes \epsilon^* L^T )$$ \label{34}
and the modified Euler characteristic as 
$$\chi_\mE (\mF)= P_\mE (\mF, 0).$$

\begin{proposition}
	Suppose $\normalfont{\text{gcd}}(a,b)=1$. Then
	$$\begin{array}{ll}
	P_\mE ((m,n), T) & \ds = \left( \frac{b^2a^2r}{2p^2q^2} + \frac{b^2a^2}{pq}  \right) T^2 + \Big( \frac{ab}{2pq}(a+b+r+2m-1+ab) + ab(n+1) \Big) T  \\
	& \ds + \frac{1+n}{2}(a+b+2m+ab-1-nr).
	\end{array}$$
\end{proposition}

\begin{proof}
	To prove the proposition, we note that:
	\begin{itemize}
		\item  $\ds \sum_{k=0}^{ab-1} \sum_{l=1}^{q-1} \frac{\omega_p^{(m+k)l}}{1-\omega_p^{-al}}=  \sum_{l=1}^{q-1} \frac{\sum_{k=0}^{ab-1} \omega_p^{(m+k)l}}{1-\omega_p^{-al}}= 0$,
		since $p \mid ab$. 
		\item  $\ds \sum_{k=0}^{ab-1} \sum_{\tiny{\begin{array}{c} l=1 \\ \frac{b}{p} \nmid l \end{array}}}^{b-1} \frac{\omega_b^{(m+k)l}}{1-\omega_b^{-al}} = 0 $.
	\end{itemize}  
	Then the result follows.  
\end{proof}

\begin{proposition}\label{35}
	$$\begin{array}{l}
	P_\mE (\left[ \mO_{P_1} \otimes \hat{\mu}_b^i \right], T) =  P_\mE (\left[ \mO_{P_4} \otimes \hat{\mu}_b^i \right], T) = a, \\
	P_\mE (\left[ \mO_{P_2} \otimes \hat{\mu}_a^i \right], T)= P_\mE (\left[ \mO_{P_3} \otimes \hat{\mu}_a^i \right], T)= b.  
	\end{array}$$
\end{proposition}

\begin{proof}
	Recall from \hyperref[31]{Proposition \ref*{31}} that $[\mO_{P_1} \otimes \hat{\mu}_b^i ]=g^i+g^{a+i} h-g^{a+i}-g^i h$. Hence 
	\begin{multline*}
	P_\mE ([ \mO_{P_1} \otimes \hat{\mu}_b^i ] , T )=P_\mE ((-i,0), T) + P_\mE ((-a-i,-1), T) \\
	- P_\mE ((-a-i,0), T) - P_\mE ((-i,-1), T)= a.
	\end{multline*}	
	Similarly, we can obtain the other results.
\end{proof}

Generally, if there is a $S$-family such that $\hat{F}_2=\hat{F}_3=\hat{F}_4=0$ and $\hat{F}_1$ consists of a single space $\mC$ with $\mu_b$-weight $i$ at the position $(k/a+A_1)(a,0) + A_2(0,1)$, then the $K$-group class of the corresponding sheaf is $(1-g^a)(1-h)g^i$ and $  P_\mE (  \mO_{P_1} \otimes L_{(k+ a A_1 , A_2, i-k-a A_1 + r A_2,0)} , T) =a. $

Thus the modified Hilbert polynomial of a sheaf corresponding to a single space $\mC$ in one chart only depends on the chart itself.

We will now look at the modified Hilbert polynomial of indecomposable locally free toric sheaves of rank $2$ on the Hirzebruch orbifold $\mH_{r}^{ab}$.

\label{3.4} 

Recall that a necessary condition for such a sheaf to be indecomposable is exactly one nonzero box summand for each chart. In this case, we set 
$$\begin{array}{c}
B_1=k+aA_1, B_2=A_2, B_3=j+bA_3, B_4=A_4 \\
\Lambda_1=a \Delta_1, \Lambda_2=\Delta_2, \Lambda_3=b \Delta_3, \Lambda_4=\Delta_4
\end{array}.$$
A locally free toric sheaf of such kind is entirely determined by $B_1,B_2,B_3, B_4 \in \mZ$, $\Lambda_1,\Lambda_2,\Lambda_3, \allowbreak \Lambda_4 \in \mZ_{\geq 0}$ such that $a \mid \Lambda_1, b \mid \Lambda_3$ and $P_1,P_2,P_3,P_4  \subset \mC^2$. It is indecomposable if and only if it satisfies one of the conditions in the first case of \hyperref[27]{Example \ref*{27}}

\begin{proposition} \label{36}
	Let $\mF$ be a locally free toric sheaf of rank $2$ with exactly one nonzero box summand for each chart. Then the modified Hilbert polynomial of $\mF$ is given by 
	$$\begin{array}{ll}
	&   P_\mE \left( (-B_1-B_3-B_4r, -B_2-B_4) \right)  \\   + &   P_\mE \left( (-B_1-\Lambda_1-B_3-\Lambda_3-B_4r-\Lambda_4r, -B_2-\Lambda_2 
	-B_4-\Lambda_4) \right) \\ 
	- &   (1-\delta_{P_1 P_2}) \Lambda_1 \Lambda_2 - (1-\delta_{P_2 P_3}) \Lambda_2 \Lambda_3 - (1-\delta_{P_3 P_4}) \Lambda_3 \Lambda_4 - (1-\delta_{P_4 P_1}) \Lambda_4 \Lambda_1.
	\end{array}$$
	where $\delta_{P_i P_j}$ is $1$ if $P_i = P_j$ and $0$ if $P_i \neq P_j$.
\end{proposition}

\begin{proof}
	We can define another toric sheaf $\mG$ such that its $S$-family $\hat{G}$ satisfies
	$$\dim(_b G_i(l_1,l_2)_m)=\dim(_b F_i(l_1,l_2)_m)$$
	for all charts. Then according to \cite[Lemma 7.7]{GJK17}, $[\mF]=[\mG] \in K_0(\mH_{r}^{ab}).$
	
	To define the $S$-family $\hat{G}$, we set $$ _b G_i(l_1, l_2) := \, _b L_{(B_1,B_2, B_3, B_4),i} (l_1,l_2) \oplus \, _b  L_{(B_1+\Lambda_1,B_2+\Lambda_2,B_3+\Lambda_3,B_4+\Lambda_4),i} (l_1,l_2) $$
	in the following regions
	$$\begin{array}{ll}
	l_1 \geq A_i + \Delta_i  \text{ or } l_2 \geq A_{i+1} + \Delta_{i+1}, \\
	l_1 < A_i + \Delta_i  \text{ and } l_2 < A_{i+1} + \Delta_{i+1}, \text{ if } P_i = P_{i+1}
	\end{array}$$
	for $1 \leq i \leq 4$. Note that if $P_i \neq P_{i+1}$, then a rectangle of size $\Delta_i \Delta_{i+1}$ is removed. Hence the modified Hilbert polynomial is decreased by $ \Lambda_i \Lambda_{i+1}$.
\end{proof}

\section{Moduli Space of Sheaves} 

\subsection{Moduli Functor} ~\\

Suppose the modified Hilbert polynomial of a pure coherent sheaf $\mF$ of dimension $d$ is 
$$P_\mE(\mF, T)= \sum_{i=0}^{d} \alpha_{\mE,i}(\mF) \frac{T^i}{i!}.$$
Then the reduced modified Hilbert polynomial is defined as 
$$p_\mE(\mF, T) = \ds \frac{P_\mE(\mF, T)}{\alpha_{\mE,d}(\mF)}$$ 
and the slope of $\mF$ is defined as 
$$\mu_{\mE}(\mF)=\frac{\alpha_{\mE,d-1}}{\alpha_{\mE,d}}.$$

\begin{definition} 
	$\mF$ is Gieseker-stable if $p_{\mE}(\mF') < p_{\mE}(\mF)$ for every proper subsheaf $\mF' \subset \mF$ \cite{Nir08}. 
\end{definition} 
\begin{definition} 
	$\mF$ is $\mu$-stable if $\mu_{\mE}(\mF') < \mu_{\mE}(\mF)$ for every proper subsheaf $\mF' \subset \mF$.  
\end{definition}

For toric varieties or orbifolds, we only need to check all the equivariant subsheaves for stability. It was proved for Gieseker stability in \cite{Koo11}. For $\mu$-stability, it was only shown in \cite{Koo11} for equivariant reflexive sheaves and recently extended to equivariant torsion-free sheaves in \cite{BDGP18}.
\label{4.1} 

We can then define a moduli functor $\underline{\mM}_{P_\mE}^{s}$, where $\underline{\mM}_{P_\mE}^{s}(S)$ is the set of equivalent classes of   $S$-flat families of Gieseker stable torsion-free sheaves on the Hirzebruch orbifold $\mH_{r}^{ab}$ with the modified Hilbert polynomial $P_\mE$. It is shown in \cite{Nir08} that there exists a quasi-projective scheme $\mM_{P_\mE}^s$ that corepresents $\underline{\mM}_{P_\mE}^{s}$ and is indeed a coarse moduli space. The closed points of $\mM_{P_\mE}^s$ are therefore in bijection with isomorphism classes of Gieseker stable torsion free sheaves on $\mH_{r}^{ab}$ with the modified Hilbert polynomial $P_\mE$.

We also define a moduli functor $\underline{\mM}_{P_\mE}^{\mu s} \subset \underline{\mM}_{P_\mE}^{s}$ which only consists of $\mu$-stable locally free shaves. The coarse moduli space is an open subset $\mM_{P_\mE}^{\mu s} \subset \mM_{P_\mE}^s$.

To get similar results of \cite[Theorem 4.15]{Koo11}, we need to modify the definition of the characteristic function for $\mH_{r}^{ab}$ and match the GIT stability with the Gieseker stability.

\begin{definition}
	Suppose $\normalfont{\text{gcd}}(a,b)=1$. Let $\mF$ be a torsion free sheaf on the Hirzebruch orbifold $\mH_{r}^{ab}$. The characteristic function $\vec{\chi}_\mF$ is defined as the disjoint union
	$$\vec{\chi}_\mF = \coprod_{k=0}^{a-1} \coprod_{j=0}^{b-1} \, _{(k,j)} \vec{\chi}_\mF. $$ 
	where
	$ _{(k,j)} \vec{\chi}_\mF: (\mZ^2)^4  \to \mZ^4 $ is the characteristic function
	$$\begin{array}{rl}
	& \left( _{(k,j)} \chi_\mF^{\sigma_1}(m_1), \, _{(k,j)} \chi_\mF^{\sigma_2}(m_2), \, _{(k,j)} \chi_\mF^{\sigma_3}(m_3), \, _{(k,j)} \chi_\mF^{\sigma_4}(m_4)      \right) \\[0.5em]
	=  &  \left(\dim_\mC( _{(k/a,0)}F_{m_1}^{\sigma_1}), \dim_\mC( _{(0,j/b)} F_{m_2}^{\sigma_2}), \dim_\mC( _{(j/b,0)}  F_{m_3}^{\sigma_3}) 
	, \dim_\mC( _{(0,k/a)} F_{m_4}^{\sigma_4}) \right).
	\end{array}
	$$
	restricted to the following box summand
	$$b_1=(k/a,0), b_2=(0,j/b), b_3=(j/b,0), b_4=(0,k/a).$$
\end{definition} 

Let $\mF$ be a torsion free sheaf of rank $2$, then $\mF$ is $\mu$-stable if and only if $\mF^{**}$ is $\mu$-stable. Since $\mF^{**}$ is locally free, indecomposability of $\mF^{**}$ implies that $S$-family $_{b_i} \hat{F}^{**}_i \neq 0$ for only one box element by \hyperref[27]{Example \ref*{27}}. Hence  $_{b_i} \hat{F}_i \neq 0$ for the same $b_i$. As a result, the characteristic function of a stable sheaf $\mF$ must be of the form 
$$
\ds \vec{\chi}_\mF =  \, _{(k,j)} \vec{\chi}_\mF.
$$

Denote by $Gr(m,n)$ the Grassmannian of $m$-dimensional subspaces of $\mC^n$. We define the following ambient quasi-projective variety:
$$
\mathcal{A} = \ds \coprod_{k=0}^{a-1} \coprod_{j=0}^{b-1} \left( \prod_{i=1}^{4} \prod_{m_i \in \mZ^2}  Gr(_{(k,j)} \chi_\mF^{\sigma_i}(m_i), 2) \right) .
$$
Then there is a locally closed subcheme $\cN_{\vec{\chi}}$ of $\mathcal{A}$ whose closed points are framed \cite{Koo10} torsion-free $S$-families with characteristic function $\vec{\chi}$. Consider the special linear group $G=SL(2,\mC)$.
Then $G$ acts regularly on $\mathcal{A}$ leaving $\cN_{\vec{\chi}}$ invariant. For any $G$-equivariant line bundle $\mL \in \text{Pic}^G(\cN_{\vec{\chi}})$, we can define the GIT stability with respect to $\mL$ \cite{Dol03}. Denote by $\cN_{\vec{\chi}}^s$ the $G$-invariant open subset of GIT stable points. We obtain a geometric quotient $\pi: \cN_{\vec{\chi}}^s \to \mM_{\vec{\chi}}^s=\cN_{\vec{\chi}}^s/G$.

\begin{proposition}
	Let $\vec{\chi}$ be the characteristic function of a torsion free sheaf of rank $2$ on the Hirzebruch orbifold $\mH_{r}^{ab}$. Let $P_\mE$ be the modified Hilbert polynomial with respect to the ample sheaf $L=\mO(\frac{b}{p} D_1 + \frac{ba}{pq} D_4)$ and the generating sheaf $\mE = \bigoplus_{k=0}^{ab-1} (-k,0)$. Then there exists an ample equivariant line bundle $\mL_{\vec{\chi}} \in \normalfont{\text{Pic}}^G(\cN_{\vec{\chi}})$ such that any torsion free sheaf $\mF$ on $\mH_{r}^{ab}$ with characteristic function $\vec{\chi}$ is Gieseker stable if and only if it is GIT stable w.r.t. $\mL_{\vec{\chi}}$. 
\end{proposition}

\begin{proof}
	If $\vec{\chi}_\mF =  \, _{(k,j)} \vec{\chi}_\mF$, then the $S$-family has exactly one nonzero box summand for each chart. Hence the double filtrations are similar to the cases of toric varieties as in \cite{Koo11} and the proof carries over without any difficulties. 
\end{proof}

\begin{remark}
	For locally free sheaves of rank $2$, we can also match the $\mu$-stability with the GIT stability w.r.t some line bundle $\mL_{\vec{\chi}}^\mu$. But in general, the line bundle $\mL_{\vec{\chi}}^\mu$ is different from $\mL_{\vec{\chi}}$. We denote the GIT quotient w.r.t this line bundle by $\mM_{\vec{\chi}}^{\mu s}$.
\end{remark}

Suppose $\mF$ is a $\mT$-equivariant sheaf on the Hirzebruch orbifold $\mH_{r}^{ab}$. By tensoring a character of $\mT$, the equivariant structure is changed, but not the underlying sheaf. This degree of freedom can be fixed by requiring $B_3=B_4=0$. In this case, we call $\vec{\chi}_\mF$ \textit{gauge-fixed}. Note that our definition is slightly different from \cite{Koo11} as we choose $B_3, B_4$ from $\sigma_4$, which has the largest index, to make the calculation easier. 

By \cite{Koo11}, the Hilbert polynomial of a torsion free toric sheaf on a smooth toric variety is fully determined by the characteristic function of
that sheaf. The result also applies to the Hirzebruch orbifold $\mH_{r}^{ab}$. Therefore, we can write $\mX_{P_\mE}$ for the set of characteristic functions with the modified Hilbert polynomial $P_\mE$. 

Since the $\mT$-action lifts naturally to $M_{P_\mE}^s$, we get the following two theorems similar to \cite{Koo11}.

\begin{theorem}
	For any choice of a generating sheaf $\mE$ with an equivariant structure on the Hirzebruch orbifold $\mH_{r}^{ab}$, there is a canonical isomorphism
	$$(\mM_{P_\mE}^s)^T \cong \coprod_{\vec{\chi} \in (\mX_{P_\mE})^{\text{gf}}} \mM_{\vec{\chi}}^s.$$
\end{theorem} 

Since (geometrically) $\mu$-stability and locally freeness are open properties for the moduli functor $\underline{\mM}_{P_\mE}^{s}$ \cite{Koo11} \cite{HL10}, we obtain 

\begin{theorem} 
	For any choice of a generating sheaf $\mE$ with an equivariant structure on the Hirzebruch orbifold $\mH_{r}^{ab}$, there is a canonical isomorphism 
	$$(\mM_{P_\mE}^{\mu s})^T \cong \coprod_{\vec{\chi} \in (\mX_{P_\mE})^{\text{gf}}} \mM_{\vec{\chi}}^{\mu s} .$$
\end{theorem}

\subsection{Generating Functions} ~\\

\label{4.2} 
Denote the moduli scheme of $\mu$-stable torsion free, resp. locally free, sheaves of rank $R$ with first Chern class $c_1$ and modified Euler characteristic $\chi_\mE$ by $M_{\mH_r^{ab}}(R,c_1,\chi_\mE)$, resp. $M_{\mH_r^{ab}}^{\text{vb}}(R,c_1,\chi_\mE)$. Our goal is to use the idea of fixed point loci to compute the following two generating functions:
$$
\sum_{\chi_\mE \in \mZ} e(M_{\mH_r^{ab}}(R,c_1,\chi_\mE)) \mq^{\chi_\mE}, \quad 
\sum_{\chi_\mE \in \mZ} e(M_{\mH_r^{ab}}^{\text{vb}} (R,c_1,\chi_\mE)) \mq^{\chi_\mE}
$$
for $R=1, 2$ with fixed $c_1$.

\subsubsection{Rank 1} ~\\

Consider $\mu$-stable torsion free toric sheaves of rank $1$ on the Hirzebruch orbifold $\mH_r^{ab}$ with fixed first Chern class $c_1=m\frac{x}{a}+ny$
where $x=c_1(\mD_{\rho_1}), y=c_1(\mD_{\rho_2})$. Let 
$$\mathsf{G}_{c_1}(\mq) = \sum_{\chi_\mE \in \mZ} e(M_{\mH_r^{ab}}(1,c_1, \chi_\mE)) \mq^{\chi_\mE}$$ 
be the generating function. Note that $e(M_{\mH_r^{ab}}(1,c_1, \chi_\mE)=e (M_{\mH_r^{ab}}(1,c_1, \chi_\mE)^\mT)$ by torus localization.  

\begin{proposition}
	$$\mathsf{G}_{m\frac{x}{a}+ny}(\mq)= \mq^{\frac{1+n}{2}(a+b+2m+ab-1-nr)} \prod_{k=1}^{\infty} \frac{1}{(1-\mq^{-ak})^2 (1-\mq^{-bk})^2 }. $$
\end{proposition}

\begin{proof}
	An equivariant line bundle $L_{(B_1,B_2,B_3,B_4)}$ is non-equivariantly trivial if and only if 
	$$B_1 + B_3 + r  B_4 =0 ; B_2 +B_4=0.$$
	If $\mF$ is a torsion free toric sheaf of rank $1$, then $\mF \otimes L_{(B_3+ B_4 r,B_4, -B_3,- B_4)}$ is gauge-fixed. Therefore, we only consider torsion free toric sheaves of rank $1$ with reflexive hulls $L_{(B_1, B_2, 0, 0)}$.
	
	For fixed $c_1$, the reflexive hull is uniquely determined as $L_{(-m, -n, 0, 0)} \cong (m,n)$. The modified Euler characteristic is given by 
	$$\chi_\mE ((m,n))= \frac{1+n}{2}(a+b+2m+ab-1-nr). $$
	
	For a torsion free toric sheaf $\mF$ with the reflexive hull $L_{(-m, -n, 0, 0)}$, the cokernel sheaf $\cQ$ of the exact sequence
	$$0 \to \mF \to L_{(-m, -n, 0, 0)} \to \cQ \to 0$$
	can be described by young diagrams. By \hyperref[35]{Proposition \ref*{35}},  
	the modified Euler characteristic of $\cQ$ increases by $a$, resp. by $b$ for each cell in the young diagrams on charts $\mU_1$ and $\mU_4$, resp. $\mU_2$ and $\mU_3$. Hence the closed points of $M_{\mH_r^{ab}}(1,c_1, \chi_\mE)^T$ are in bijection with four partitions $(\lambda_1, \lambda_2, \lambda_3, \lambda_4)$ such that
	$$\frac{1+n}{2}(a+b+2m+ab-1-nr) - a(\# \lambda_1 + \# \lambda_4) - b(\# \lambda_2 + \# \lambda_3)= \chi_\mE.$$ 
\end{proof} 

\begin{remark}
	By \hyperref[35]{Proposition \ref*{35}} and \hyperref[31]{Proposition \ref*{31}}, the modified Euler characteristic of $\cQ$ is independent of the fine grading, whereas the $K$-group class is not. Hence we do not need to consider the colored Young diagrams as in \cite{GJK17}. 
\end{remark}

\subsubsection{Rank 2} ~\\

For a toric surface, there is a nice expression that relates the generating functions of torsion free and locally free sheaves given by \cite{Got99}. We also derive a similar relation for the Hirzebruch orbifold $\mH_{r}^{ab}$, which is given in \hyperref[1]{Theorem \ref*{1}}.

\begin{proof}[Proof of Theorem 0.1]
	The proof is similar to that of \cite[lemma 7.4]{GJK17} except in our case the moduli scheme is stratified by the modified Euler characteristics.
\end{proof} 

Let $\mF$ be a locally free toric sheaf of rank $2$ on the Hirzebruch orbifold $\mH_r^{ab}$. By tensoring with $L_{(B_3+ r B_4 , B_4, -B_3,- B_4)}$, we only consider toric sheaves with $B_3=B_4=0$, which are gauge-fixed. 

From \hyperref[27]{Example \ref*{27}}, we know that there are three types of indecomposable toric sheaves. Hence, the connected components of the fixed locus $M_{\mH_r^{ab}}^{\text{vb}} (R,c_1,\chi_\mE)^\mT$ can be explicitly classified as follows:

\begin{enumerate}
	\item $\Lambda_i > 0$ for all $i$. $P_i$'s are mutually distinct. \footnote{For notation, see \hyperref[3.4]{Section \ref*{3.4}} and \hyperref[27]{Example \ref*{27}}}
	
	Consider four equivariant line bundles $L_1, L_2, L_3, L_4 \subset \mF$ generated by $P_1, P_2, P_3, P_4$ respectively. 
	$$\begin{array}{ll}
	L_1 = L_{B_1, B_2 + \Lambda_2, \Lambda_3 , \Lambda_4}, &  L_2 = L_{B_1+\Lambda_1, B_2 , \Lambda_3 , \Lambda_4}, \\
	L_3 = L_{B_1+\Lambda_1, B_2+\Lambda_2 , 0 , \Lambda_4}, &
	L_4 = L_{B_1+\Lambda_1, B_2+\Lambda_2 , \Lambda_3 , 0}. 
	\end{array}$$
	Any rank $1$ equivariant subsheaf of $\mF$ is contained in one of $L_i$ and does not have bigger slope. Hence it suffices to test $\mu_{\mE}(L_i) < \mu_{\mE}(\mF)$ for all $L_i$. The stability conditions are given by
	$$\begin{array}{ll}
	\Lambda_1 < pq \Lambda_2 + \Lambda_3 + (r +pq) \Lambda_4, &
	pq \Lambda_2 < \Lambda_1 + \Lambda_3 + (r +pq) \Lambda_4, \\
	\Lambda_3 < \Lambda_1 + pq \Lambda_2 + (r +pq) \Lambda_4, &
	(r +pq) \Lambda_4 < \Lambda_1 + pq \Lambda_2 + \Lambda_3.
	\end{array}$$
	
	Denote by $D$ the set of points $(P_1, P_2, P_3, P_4) \in (\mP^1)^4$ where $P_1, P_2, P_3, P_4$ are mutually distinct. Then the connected component of the fixed locus is given by $D / \text{SL}(2,\mC)$ and $e(D / \text{SL}(2,\mC))=e(\mP^1-\{0,1,\infty\})=-1$.
	
	\item $\Lambda_{i'} = 0$ for some unique $i'$ and $\Lambda_i > 0$ for $i \neq i'$. $P_{i'}$ is omitted and $P_i$'s are mutually distinct for $i \neq {i'}$.
	
	Suppose $\Lambda_1$ is $0$, then the above inequalities are reduced to $$\begin{array}{ll}
	pq \Lambda_2 < \Lambda_3 + (r +pq) \Lambda_4, &
	\Lambda_3 < pq \Lambda_2 + (r +pq) \Lambda_4, \\
	(r +pq) \Lambda_4 < pq \Lambda_2 + \Lambda_3. &
	\end{array}$$ 
	Hence the connected component is $D / \text{SL}(2,\mC)$, where $D$ is the set of points $(P_2, P_3, P_4) \in (\mP^1)^3$ where $P_2, P_3, P_4$ are mutually distinct, and  $e(D / \text{SL}(2,\mC))=1$.
	
	\item $\Lambda_i > 0$ for all $i$. Only two of $P_i$'s are same.     	    
	
	Suppose $P_1 =P_2$, $P_3, P_4$ are mutually distinct. Then we need to consider line bundles $L_1', L_3, L_4$ where $L_1' = L_{(B_1, B_2, \Lambda_3 , \Lambda_4)})$. The stability conditions are are given by
	$$\begin{array}{ll}
	\Lambda_1 + pq \Lambda_2 < \Lambda_3 + (r +pq) \Lambda_4, &
	\Lambda_3 < \Lambda_1 + pq \Lambda_2 + (r +pq) \Lambda_4, \\
	(r +pq) \Lambda_4 < \Lambda_1 + pq \Lambda_2 + \Lambda_3. &
	\end{array}$$  	
	Similar to the case $2$, the topological Euler number of this component is $1$.
\end{enumerate}

Thus there are 11 types of incidence spaces contributing to the generating function similar to the case of Hirzebruch surface in \cite{Koo15}. 

Consider locally free toric sheaves of rank $2$ with fixed first Chern class $c_1= \frac{m}{a}x +n y$ where $c_1(\mD_{\rho_1})=x, c_1(\mD_{\rho_2})=y$. By \hyperref[36]{Proposition \ref*{36}}, one can show that
$$c_1= -(2B_1+\Lambda_1+\Lambda_3+\Lambda_4 r) \frac{x}{a} - (2B_2+\Lambda_2+\Lambda_4) y. $$
Hence
$$
2B_1+\Lambda_1+\Lambda_3+\Lambda_4 r = -m, \quad  
2B_2+\Lambda_2+\Lambda_4 =-n.
$$

If $\mF$ is of the first type mentioned above, then the modified Euler characteristic is given by 
$$\begin{array}{rl}
&  P_\mE \left( (-B_1, -B_2) , 0 \right)+ P_\mE \left( (-B_1-\Lambda_1-\Lambda_3-\Lambda_4r, -B_2-\Lambda_2-\Lambda_4) , 0\right) \\
&  - \Lambda_1 \Lambda_2 - \Lambda_2 \Lambda_3 - \Lambda_3 \Lambda_4 - \Lambda_4 \Lambda_1 \\
= &  \ds \frac{1}{2}(C-r)n + C + m  + \frac{mn}{2} - \frac{n^2 r}{4} - \frac{1}{2} (\Lambda_2+\Lambda_4) (\Lambda_1+\frac{r}{2}\Lambda_2+\Lambda_3-\frac{r}{2} \Lambda_4) \\
\end{array}$$
where $C=a+b+ab-1$. Similarly, we can obtain the modified Euler characteristics for other types. 

Let
$$\mathsf{H}_{c_1}^{\text{vb}}(\mq) = \sum e(M_{\mH_r^{ab}}^{\text{vb}}(2,c_1, \chi_\mE)) \mq^{\chi_\mE}$$
be the generating function. Define $f= \frac{1}{2}(C-r)n + C + m  + \frac{mn}{2} - \frac{n^2 r}{4}.$ Then

\begin{align*}
& \ds \mathsf{H}_{\frac{m}{a}x +n y}^{\text{vb}}(\mq) =  -\sum_{\tiny{ \begin{array}{c}
		\Lambda_1, \Lambda_2, \Lambda_3, \Lambda_4 \in \mZ_{>0} \\ 
		a \mid \Lambda_1, b \mid \Lambda_3\\
		2 \mid -m-\Lambda_1-\Lambda_3-r\Lambda_4 \\ 
		2 \mid -n-\Lambda_2-\Lambda_4 \\            
		\Lambda_1 < pq \Lambda_2 + \Lambda_3 + (r +pq) \Lambda_4 \\
		pq \Lambda_2 < \Lambda_1 + \Lambda_3 + (r +pq) \Lambda_4 \\
		\Lambda_3 < \Lambda_1 + pq \Lambda_2 + (r +pq) \Lambda_4 \\
		(r +pq) \Lambda_4 < \Lambda_1 + pq \Lambda_2 + \Lambda_3
		\end{array} } } \mq^{ f - \frac{1}{2} (\Lambda_2+\Lambda_4) (\Lambda_1+\frac{r}{2}\Lambda_2+\Lambda_3-\frac{r}{2} \Lambda_4)}   \displaybreak \\
& \ds + \sum_{\tiny{ \begin{array}{c}
		\Lambda_1, \Lambda_2, \Lambda_3, \Lambda_4 \in \mZ_{>0} \\ 
		a \mid \Lambda_1, b \mid \Lambda_3\\
		2 \mid -m-\Lambda_1-\Lambda_3-r\Lambda_4 \\ 
		2 \mid -n-\Lambda_2-\Lambda_4 \\            
		\Lambda_1 + \Lambda_3 <  pq \Lambda_2 + (r +pq) \Lambda_4\\
		pq \Lambda_2 < \Lambda_1 + \Lambda_3 + (r +pq) \Lambda_4 \\
		(r +pq) \Lambda_4 < \Lambda_1 + pq \Lambda_2 + \Lambda_3
		\end{array} } }
\mq^{f - \frac{1}{2} (\Lambda_2+\Lambda_4) (\Lambda_1+\frac{r}{2}\Lambda_2+\Lambda_3-\frac{r}{2} \Lambda_4)} + 5 \text{ similar terms}   \\
& \ds +\sum_{\tiny{ \begin{array}{c}
		\Lambda_2, \Lambda_3, \Lambda_4 \in \mZ_{>0} \\ 
		b \mid \Lambda_3\\
		2 \mid -m-\Lambda_3-r\Lambda_4 \\ 
		2 \mid -n-\Lambda_2-\Lambda_4 \\            
		pq \Lambda_2 <   \Lambda_3 + (r +pq) \Lambda_4 \\
		\Lambda_3 < pq \Lambda_2 + (r +pq) \Lambda_4 \\
		(r +pq) \Lambda_4 <  pq \Lambda_2 + \Lambda_3
		\end{array} } } \mq^{f - \frac{1}{2} (\Lambda_2+\Lambda_4) (\frac{r}{2}\Lambda_2+\Lambda_3-\frac{r}{2} \Lambda_4)} +  3 \text{ similar terms}. \\ 
\end{align*}
Note that the first term corresponds to the component of the first type and the negative sign comes from  $e(\mP^1-\{0,1,\infty\})=-1$. The signs for the remaining terms are positive because the topological Euler number is $1$ for the other components.
Using proper substitutions, we can simplify this generating function further.
\begin{proposition}
	Suppose $\normalfont{\text{gcd}}(a,b)=1$. Let $ f=\frac{1}{2}(C-r)n+C+m+\frac{mn}{2}-\frac{n^2 r}{4}$ where $C=a+b+ab-1$. If $r \geq 0$, the generating function $\mathsf{H}_{\frac{m}{a}x +n y}^{\text{vb}}(\mq)$ equals 
	\begin{align*}
	\ds \mathsf{H}_{\frac{m}{a}x +n y}^{\text{vb}}(\mq)  = &   
	\ds \Bigg( - \sum_{C_1} + \sum_{C_6} + \sum_{ C_7 }+  \sum_{ C_8 } +
	\sum_{C_9} 		
	\Bigg) \mq^{f - \frac{1}{2} j (i+\frac{r}{2}j)} 
	\\	&	+  \Bigg( \sum_{C_2} + \sum_{C_3} + 	\sum_{C_4} +	\sum_{C_5} \Bigg)
	\mq^{f - \frac{1}{4}ij  + \frac{1}{4}jk - \frac{1}{4}kl - \frac{1}{4}li -  \frac{r}{4} l^2} 
	\end{align*}
	where
	\begin{align*}
	C_1=\{ & (i,j,k,l) \in \mZ^4: 2 \mid m+i, 2 \mid n+j, 2 \mid j-l, 2a \mid i+k+r(j-l)  , \\ &  2b \mid i-k , i = pqj, 
	-j < l < j,  -pqj-r(j-l)<k<pqj \}, \\
	C_2=\{ & (i,j,k,l) \in \mZ^4: 2 \mid m+i, 2 \mid n+j, 2 \mid j-l, 2a \mid i+k+  r(j+l), \\
	&  2b \mid i-k ,  k < pql < i,   l < j, -i-r(j+l) < k , -pqj-r(j+l) < k \},  \\
	C_3=\{ & (i,j,k,l) \in \mZ^4: 2 \mid m+i, 2 \mid n+j, 2 \mid j-l,  2b \mid i+k +r(j-l),\\
	&  2a \mid i-k,  k < pql < i, l < j, -i-r(j+l) < k , -pqj-r(j+l) < k\},    \\
	C_4=\{ & (i,j,k,l) \in \mZ^4: 2 \mid m+i, 2 \mid n+j, 2 \mid j-l, 2b \mid i+k    -r(j-l) , \\
	&  2a \mid i-k,  k < pql < i, l < j, -i+r(j-l) < k , -pqj < k\},  \\
	C_5=\{ & (i,j,k,l) \in \mZ^4: 2 \mid m+i, 2 \mid n+j, 2 \mid j-l, 2a \mid i+k  -r(j-l) ,  \\
	&  2b \mid i-k, k < pql < i, l < j, -i+r(j-l) < k , -pqj < k \}, \\
	C_6=\{ & (i,j,k) \in \mZ^3: 2 \mid m+i, 2 \mid n+j, 2 \mid j+k,  2b \mid 2i+r(j+k) ,  \\
	&  -\frac{r}{2}(j+k) < i  , i < pqj  -\frac{i}{r+pq} - \frac{rj}{r+pq} < k < \frac{i}{pq} \},  \\
	C_7=\{ & (i,j,k) \in \mZ^3: 2 \mid m+i, 2 \mid n+j, 2 \mid j+k,  2a \mid 2i+r(j+k) ,   \\
	&  -\frac{r}{2}(j+k) < i  ,  i < pqj -\frac{i}{r+pq} - \frac{rj}{r+pq} < k < \frac{i}{pq} \}, \displaybreak  \\
	C_8=\{ & (i,j,k) \in \mZ^3: 2 \mid m+i, 2 \mid n+j, 2a \mid i+k+ 2 rj,  2b \mid i-k  , \\ & -pqj- 2rj < k < pqj < i \},\\
	C_9=\{ & (i,j,k) \in \mZ^3: 2 \mid m+i, 2 \mid n+j, 2a \mid i+k,  2b  \mid i-k, \\ &    -pqj<k<pqj < i  \},
	\end{align*}
\end{proposition}

\begin{proof}
	Set $i=\Lambda_1+\Lambda_3-r\Lambda_4, j=\Lambda_2+\Lambda_4, k=\Lambda_1-\Lambda_3-r\Lambda_4, l=\Lambda_2-\Lambda_4$. The first term is split into two 
	$$\begin{array}{l}
	\ds \sum_{\tiny{ \begin{array}{c}
			i,j,k,l \in \mZ \\ 
			2 \mid m+i, 2 \mid n+j, 2 \mid j-l \\ 
			2b \mid i-k, 2a \mid i+k+r(j-l)  \\            
			pqj \leq i, -j < l < j\\
			-pqj-r(j-l)<k<pqj
			\end{array} } } \mq^{f  - \frac{1}{2} j (i+\frac{r}{2}j)} 
	\ds  \sum_{\tiny{ \begin{array}{c}
			i,j,k,l \in \mZ \\ 
			2 \mid m+i, 2 \mid n+j, 2 \mid j-l \\ 
			2b \mid i-k, 2a \mid i+k+r(j-l)  \\            
			i < pqj, -i < pql < i+r(j-l)\\
			-i-r(j-l) < k < i
			\end{array} } } \mq^{f - \frac{1}{2} j (i+\frac{r}{2}j)} 
	\end{array}$$
	based on whether $pqj \leq i$ or $pqj > i$. By same substitutions, the first three terms can be combined into one. The remaining terms can be obtained by the following substitution. 
	
	$$\begin{tabular}{|c|l|}
	\hline
	Term & Substitutions \\ \hline
	$4$th & $i=\Lambda_1+\Lambda_3-r\Lambda_4, j=\Lambda_2+\Lambda_4, k=\Lambda_1-\Lambda_3-r\Lambda_4, l=\Lambda_4-\Lambda_2$\\\hline
	$5$th & $i=\Lambda_1+\Lambda_3-r\Lambda_4, j=\Lambda_2+\Lambda_4, k=-\Lambda_1+\Lambda_3-r\Lambda_4, l=\Lambda_4-\Lambda_2$ \\\hline
	$6$th & $i=\Lambda_1+\Lambda_3+r\Lambda_4, j=\Lambda_2+\Lambda_4, k=-\Lambda_1+\Lambda_3+r\Lambda_4, l=\Lambda_2-\Lambda_4$ \\\hline
	$7$th & $i=\Lambda_1+\Lambda_3+r\Lambda_4, j=\Lambda_2+\Lambda_4, k=\Lambda_1-\Lambda_3+r\Lambda_4, l=\Lambda_2-\Lambda_4$ \\ \hline
	$8$th & $i=\Lambda_3-r\Lambda_4, j=\Lambda_2+\Lambda_4, k=\Lambda_4-\Lambda_2$ \\ \hline
	$9$th & $i=\Lambda_1+\Lambda_3-r\Lambda_4, j=\Lambda_4, k=\Lambda_1-\Lambda_3-r\Lambda_4$ \\ \hline
	$10$th & $i=\Lambda_1-r\Lambda_4, j=\Lambda_2+\Lambda_4, k=\Lambda_4-\Lambda_2$ \\ \hline
	$11$th & $i=\Lambda_1+\Lambda_3, j=\Lambda_2, k=\Lambda_1-\Lambda_3$ \\ \hline
	\end{tabular} $$   
\end{proof}

If $r=0$, the above result yields the \hyperref[3]{Theorem \ref*{3}} for the orbifold $\mP(a,b) \times \mP^1$.

\begin{remark}
	If $a=b=1$, the orbifold becomes the variety $\mP^1 \times \mP^1$ and $f=\frac{mn}{2}+m+n+2$. Consider a torsion free sheaf $\mF$ of rank $2$ with $c_1= m x +n y$ where $c_1(\mD_{\rho_1})=x, c_1(\mD_{\rho_2})=y$. Suppose $c_2(\mF)= c xy$. One can show that 
	$\chi(\mF)= - c + mn + m + n +2.$
	Hence the above generating function agrees with the one given in \cite[Corollary 2.3.4]{Koo10} when $\lambda=1$. Note that the divisor $D_4$ in \cite{Koo10} is really $D_2$ in our paper, but $D_2 \sim D_4$ in the case of $\mP^1 \times \mP^1$.
\end{remark}

Let $(i,j) \in \text{Pic}(\mH_{r}^{ab})$. One can show that tensoring $ - \otimes (i,j)$ preserves $\mu$-stability. Suppose $\mF$ is a locally free toric sheaf of rank $2$ on the Hirzebruch orbifold $\mH_{r}^{ab}$ with $c_1(\mF)= \frac{m}{a}x +n y$. Then 
$$\chi_\mE(\mF \otimes (i,j)) = \chi_\mE(\mF) +  i (2+n+2j) + j (ab+a+b-1-r+m-nr-rj).$$

Let $g(i,j) = i (2+n+2j) + j (ab+a+b-1-r+m-nr-rj)$. We obtain an isomorphism
$$M_{\mH_r^{ab}}^{\text{vb}} (2,c_1,\chi_\mE) \cong M_{\mH_r^{ab}}^{\text{vb}} (2,c_1+\dfrac{2i}{a}x + 2j y,\chi_\mE+  g(i,j)),$$
which induces
$$ \sum_{\chi_\mE \in \mZ} e(M_{\mH_r^{ab}}^{\text{vb}} (2,c_1 +\dfrac{2i}{a}x + 2j y,\chi_\mE)) \mq^{\chi_\mE} = \mq^{g(i,j)} \sum_{\chi_\mE \in \mZ} e(M_{\mH_r^{ab}}^{\text{vb}} (2,c_1,\chi_\mE)) \mq^{\chi_\mE}. $$

Thus for the Hirzebruch orbifold, the only interesting cases for the generating functions are $(m,n)= (0,0), (0,1), (1,0)$ and $(1,1)$.

\begin{proposition} \label{412} 
	Consider the orbifold $\mH_{0}^{12}$, which is $\mP(1,2) \times \mP^1$. In this case, $r=0, a=1, b=2, p=1, q=2, C=4$. Let $c_1(\mF)= mx +n y$ where $c_1(\mD_{\rho_1})=x$ and $c_1(\mD_{\rho_2})=y$.
	
	\begin{enumerate}
		\item  If $(m,n)=(0,0)$, then $f=4$. 
		\begin{align*}
		 \mathsf{H}_{0}^{\text{vb}}(\mq)   = & \ds  - \sum_{t=1}^{\infty}  (2t-1)^2 \mq^{4-4t^2}\\
		&\ds +  \sum_{t=1}^{\infty}  \sum_{u=1}^{\infty} \sum_{p=1}^{2t}  
		4 \mq^{4- (4t+4)(t-p+1)-2p-2u }  \,   \frac{\mq^{-(2u+2p)p}-\mq^{-(2u+2p)(2t+1)}}{1-\mq^{-(2u+2p)}} \\
		&\ds +  \sum_{t=1}^{\infty}   \sum_{u=1}^{\infty} \sum_{p=1}^{2t}  
		4\mq^{4-(4t+2)(t-p+1)} \,  \frac{\mq^{-(2u+2p-2)p}-\mq^{-(2u+2p-2)(2t+1)}}{1-\mq^{-(2u+2p-2)}}   \\
		&\ds +  \sum_{t=1}^{\infty}   \sum_{p=1}^{2t}  
		4\mq^{4- (4t+4)(t-p+1) - 2p} \,  \frac{\mq^{-2p^2}-\mq^{-(2t+1)(2p)} }{(1-\mq^{-(4t+4-2p)}) (1-\mq^{-2p}) }    \\
		&\ds + \sum_{t=1}^{\infty}    \sum_{p=1}^{2t-1}  
		4 \mq^{4- 2t(2t-2p+1)} \,  \frac{\mq^{-2p^2}-\mq^{-4pt} }{(1-\mq^{-(4t-2p)}) (1-\mq^{-2p}) }  \\
		&\ds + \sum_{t=1}^{\infty} 2 (2t-1) \frac{\mq^{4 - 4t(t+1)}  }{1-\mq^{-4t}} +  \sum_{t=1}^{\infty} 2 (2t-1) \frac{\mq^{4 - (4t-2)t }  }{1-\mq^{-(4t-2)}} \\
		& \ds + \sum_{t=1}^{\infty} 2 (2t-1) \frac{\mq^{4 - 4t(t+1) }  }{1-\mq^{-4t}} + \sum_{t=1}^{\infty} 2 (2t-1) \frac{\mq^{4 - 2t(2t+1) }  }{1-\mq^{-4t}} \\
		 = & \, 2\mq^2 + 5 + \frac{8}{\mq^2} + \frac{18}{\mq^4} + O\left[\frac{1}{\mq}\right]^5.
		\end{align*}	
		
		\item If $(m,n)=(1,0)$, then $f=5$.    	
		\begin{align*}
		\mathsf{H}_{x}^{\text{vb}}(\mq)  = & 
		\ds \sum_{t=1}^{\infty}  \sum_{u=1}^{\infty} \sum_{p=1}^{2t}  
		2 \mq^{5- (4 t + 1) (t - p + 1) } \, \frac{\mq^{-(2u+2p-2)p}-\mq^{-(2u+2p-2)(2t+1)}}{1-\mq^{-(2u+2p-2)}}  \\
		& \ds + \sum_{t=1}^{\infty}  \sum_{u=1}^{\infty} \sum_{p=1}^{2t}  
		2 \mq^{5-(4t+2)(t-p+1)+t+u} \, \frac{\mq^{-(2u+2p-2)p}-\mq^{-(2u+2p-2)(2t+1)}}{1-\mq^{-(2u+2p-2)}} \\
		&\ds + \sum_{t=1}^{\infty}  \sum_{u=1}^{\infty} \sum_{p=1}^{2t-1}  
		2 \mq^{5-(4t - 1)(t-p)} \, \frac{\mq^{-(2u+2p-2)p}-\mq^{-2t(2u+2p-2)}}{1-\mq^{-(2u+2p-2)}} \\
		& \ds +  \sum_{t=1}^{\infty}  \sum_{u=1}^{\infty} \sum_{p=1}^{2t}  
		2 \mq^{5-(4t+2)(t-p+1)-t-u} \, \frac{\mq^{-(2u+2p-2)p}-\mq^{-(2u+2p-2)(2t+1)}}{1-\mq^{-(2u+2p-2)}} \\
		&\ds +  \sum_{t=1}^{\infty}  \sum_{p=1}^{2t-1}  
		2 \mq^{5 - (4 t + 1) (t - p) - 2 p} \, \frac{\mq^{-2 p^2} - 
			\mq^{-4 p t} }{(1 - \mq^{-(4 t - 2 p)}) (1-\mq^{-2p}) }   \\
		& \ds + \sum_{t=1}^{\infty}  \sum_{p=1}^{2t-1}  
		2  \mq^{5-(4t+1)(t-p)-p} \, \frac{\mq^{-2p^2}-\mq^{-4pt} }{(1-\mq^{-(4t-2p)}) (1-\mq^{-2p}) }  \\
		&\ds + \sum_{t=1}^{\infty}  \sum_{p=1}^{2t-1}  
		2 \mq^{5 - (4 t + 3) (t - p) - 2 p} \, \frac{\mq^{-2p^2}-\mq^{-4pt} }{(1-\mq^{-(4t-2p)}) (1-\mq^{-2p}) }   \displaybreak  \\
		& \ds +  \sum_{t=1}^{\infty}  \sum_{p=1}^{2t-1}  
		2 \mq^{5-(4t+3)(t-p)-3p} \, \frac{\mq^{-2p^2}-\mq^{-4pt} }{(1-\mq^{-(4t-2p)}) (1-\mq^{-2p}) } \\ 
		&\ds + \sum_{t=1}^{\infty} 2 t \frac{\mq^{5 - (4 t + 1) (t + 1)}  }{1 - \mq^{-(4 t + 1)}} +  \sum_{t=1}^{\infty}  2t \, \frac{\mq^{5 - (4t-1)t}  }{1-\mq^{-(4t-1)}} + \sum_{t=1}^{\infty} 4 t \frac{\mq^{5 - (4 t + 1) t  }  }{1-\mq^{-2t}}   \\
	    =  	& \, 2\mq^3 + 4\mq^2 + 6\mq + 8 + \frac{12}{\mq} + \frac{12}{\mq^2} + \frac{14}{\mq^3} + \frac{20}{\mq^4} + O \left[\frac{1}{\mq}\right]^{5}.
		\end{align*}
		
		\item 	If $(m,n)=(0,1)$, then $f=6$.
		\begin{align*}
		\mathsf{H}_{y}^{\text{vb}}(\mq)  = & - \sum_{t=1}^{\infty}  4t^2 \mq^{6-(2t+1)^2}\\
		& \ds + \sum_{t=1}^{\infty}  \sum_{u=1}^{\infty} \sum_{p=1}^{2t-1}  
		4  \mq^{6-2t(2t-2p+1) } \, \frac{\mq^{-(2u+2p-2)p}-\mq^{-2t(2u+2p-2)}}{1-\mq^{-(2u+2p-2)}}   \\
		& \ds +  \sum_{t=1}^{\infty}  \sum_{u=1}^{\infty} \sum_{p=1}^{2t-1}  
		4\mq^{5-4t(t-p+1)-2u} \, \frac{\mq^{-(2u+2p)p}-\mq^{-2t(2u+2p)}}{1-\mq^{-(2u+2p)}} \\
		& \ds + \sum_{t=1}^{\infty}  \sum_{p=1}^{2t}  
		4  \mq^{6-(4t+2)(t-p+1)} \, \frac{\mq^{-2p^2}-\mq^{-2p(2t+1)} }{(1-\mq^{-(4t+2-2p)}) (1-\mq^{-2p}) }    \\
		& \ds + \sum_{t=1}^{\infty}  \sum_{p=1}^{2t-1}  
		4 \mq^{5 - 4t(t-p+1)} \, \frac{\mq^{-2p^2}-\mq^{-4pt} }{(1-\mq^{-(4t+2-2p)}) (1-\mq^{-2p}) }  \\
		& \ds + \sum_{t=1}^{\infty} (4t-1) \frac{\mq^{6 - 2t(2t+1)}  }{1-\mq^{-4t}} +  \sum_{t=1}^{\infty} (4t-3) \frac{\mq^{6 - (2t-1)(2t+1) }  }{1-\mq^{-(4t-2)}}  \\
		& \ds + \sum_{t=1}^{\infty} 2(2t-1) \frac{\mq^{6 - 2t(2t-1) }  }{1-\mq^{-(4t-2)}} + \sum_{t=1}^{\infty} 4t \, \frac{\mq^{6 - (2t+1)(2t+3) }  }{1-\mq^{-(4t+2)}}  \\
		= & \, 2\mq^4 + \mq^3 + 6 \mq^2 + \mq + 9 +\frac{5}{\mq} + \frac{14}{\mq^2} + \frac{5}{\mq^3} + \frac{17}{\mq^4} + O\left[\frac{1}{\mq}\right]^{5}.
		\end{align*}
		
		\item  	If $(m,n)=(1,1)$, then $f=\dfrac{15}{2}$.
		\begin{align*}
		\mathsf{H}_{x+y}^{\text{vb}}(\mq) = 
		&  \ds   \sum_{t=1}^{\infty}  \sum_{u=1}^{\infty} \sum_{p=1}^{2t}  
		2 \mq^{7-(4t+3)(t-p)-2p} \, \frac{\mq^{-(2u+2p-2)p}-\mq^{-(2u+2p-2)(2t+1)}}{1-\mq^{-(2u+2p-2)}} \\
		& \ds +  \sum_{t=1}^{\infty}  \sum_{u=1}^{\infty} \sum_{p=1}^{2t-1}  
		2 \mq^{7 - (4t-1)(t-p+1)- u + p} \, \frac{\mq^{-(2u+2p-2)p}-\mq^{-2t(2u+2p-2)}}{1-\mq^{-(2u+2p-2)}}    \\
		& \ds +  \sum_{t=1}^{\infty}  \sum_{u=1}^{\infty} \sum_{p=1}^{2t-1}  
		2 \mq^{8-(4t-1)(t-p)-2p} \, \frac{\mq^{-(2u+2p-2)p}-\mq^{-2t(2u+2p-2)}}{1-\mq^{-(2u+2p-2)}}   \\
		& \ds +  \sum_{t=1}^{\infty}  \sum_{u=1}^{\infty} \sum_{p=1}^{2t-1}  
		2 \mq^{7- (4t+1)(t-p)- p + u} \, \frac{\mq^{-(2u+2p-2)p}-\mq^{-2t(2u+2p-2)}}{1-\mq^{-(2u+2p-2)}}  \displaybreak \\
		& \ds +  \sum_{t=1}^{\infty}  \sum_{p=1}^{2t}  
		2 \mq^{8-(4t+3)(t-p+1)} \, \frac{\mq^{-2p^2}-\mq^{-2p(2t+1)} }{(1-\mq^{-(4t+2-2p)}) (1-\mq^{-2p}) }      \\
		& \ds +  \sum_{t=1}^{\infty}  \sum_{p=1}^{2t}  
		2 \mq^{8-(4t+3)(t-p+1)-p} \, \frac{\mq^{-2p^2}-\mq^{-2p(2t+1)} }{(1-\mq^{-(4t+2-2p)}) (1-\mq^{-2p}) }  \\
		& \ds +  \sum_{t=1}^{\infty}  \sum_{p=1}^{2t}  
		2 \mq^{7-(4t+1)(t-p+1)} \, \frac{\mq^{-2p^2}-\mq^{-2p(2t+1)} }{(1-q^{-(4t+2-2p)}) (1-\mq^{-2p}) }  \\
		& \ds + \sum_{t=1}^{\infty}  \sum_{p=1}^{2t}  
		2 \mq^{7-(4t+1)(t-p+1)+ p} \, \frac{\mq^{-2p^2}-\mq^{-2p(2t+1)} }{(1-\mq^{-(4t+2-2p)}) (1-\mq^{-2p}) }   \\
		& \ds + \sum_{t=1}^{\infty}  2t \, \frac{\mq^{\frac{15}{2} - \frac{1}{2}(2t+1)(4t+1)}  }{1-\mq^{-(4t+1)}} + \sum_{t=1}^{\infty}  2t \, \frac{\mq^{\frac{15}{2} - \frac{1}{2}(2t+1)(4t-1)}  }{1-\mq^{-(4t-1)}}  \\
		& \ds + \sum_{t=1}^{\infty} 2(2t-1) \, \frac{\mq^{\frac{15}{2} - \frac{1}{2}(2t-1)(4t+1) }  }{1-\mq^{-(4t-2)}} + \sum_{t=1}^{\infty} 2(2t-1) \, \frac{\mq^{\frac{15}{2} - \frac{1}{2}(2t-1)(4t-1) }  }{1-\mq^{-(4t-2)}}  \\
	=  	& \ds \, 2 \mq^6 + 4\mq^5+ 6 \mq^4 + 8 \mq^3 + 10 \mq^2 + 14 + \frac{14}{\mq} +  \frac{18}{\mq^2} +  \frac{24}{\mq^3} +  \frac{22}{\mq^4} + O \left[ \frac{1}{\mq} \right]^{5}.
		\end{align*}
	\end{enumerate}
\end{proposition}

\begin{proof}
	We will show how to rewrite the sums over $C_2$ and $C_3$ in the case of $(m,n)=(0,0)$. The calculation of other parts is similar.
	
	The second and third terms can be combined into one
	$$\sum_{C_2'} 4 \mq^{4 -  \frac{1}{4}ij  + \frac{1}{4}jk - \frac{1}{4}kl - \frac{1}{4}li}
	$$ 
	where 
	\begin{multline*}
	C_2'=\{ (i,j,k,l) \in \mZ^4: 2 \mid i, 2 \mid j, 2 \mid l, 2 \mid k, 
	4 \mid i-k , -i < k < 2l < i, -2j < k , l < j \}.
	\end{multline*}
	
	It can be then split into two terms by either $i < 2j$ or $2j \leq i$.
	$$\Big( \sum_{C_2''} + \sum_{C_3''}   \Big) 4 \mq^{4 -  \frac{1}{4}ij  + \frac{1}{4}jk - \frac{1}{4}kl - \frac{1}{4}li}
	$$
	where
	\begin{align*}
	C_2''= & \{ (i,j,k,l) \in \mZ^4: 2 \mid i, 2 \mid j, 2 \mid l, 2 \mid k,  4 \mid i-k ,   -2j <-i < k < 2l < i < 2j \}, \\
	C_3''= & \{ (i,j,k,l) \in \mZ^4: 2 \mid i, 2 \mid j, 2 \mid l, 2 \mid k,  4 \mid i-k ,    -i < -2j < k < 2l < 2j <  i  \} .
	\end{align*}	
	
	Suppose $i=4t+4$, then we have the following picture for the case $i < 2j$.
	$$   \begin{tikzpicture}
	\draw[thick] (-5,0) -- (6.5,0);
	\foreach \x in {-4,-3,-2,-1,-0.5, 0,2,3,4,6}
	\draw (\x cm,2pt) -- (\x cm,-2pt);
	\node[below] at (4,0) {$4t+4$};
	\node[above] at (4,0) {$i$};
	\node[below] at (2,0) {$2t+2$};
	\node[above] at (2,0) {$\frac{i}{2}$};
	\node[below] at (-2,0) {$-2t-2$};
	\node[above] at (-2,0) {$-\frac{i}{2}$};
	\node[below] at (-4,0) {$-4t-4$};
	\node[above] at (-4,0) {$-i$};
	\node[below] at (0,0) {$0$};
	\node[above] at (3,0) {$j$};
	\node[above] at (6,0) {$2j$};
	\node[above] at (-0.5,0) {$l$};
	\node[above] at (-1,0) {$2l$};
	\node[above] at (-3,0) {$k$};
	\end{tikzpicture}	
	$$	
	Hence $j=2t+2+2u, l=2t+2-2p, k=4t+4-4p-4s$ and 
	\begin{align*}
	\sum_{C_2'' } \mq^{\frac{1}{4}ij  -  \frac{1}{4}jk + \frac{1}{4}kl + \frac{1}{4}li} &  =  
	\sum_{t=1}^{\infty}  \sum_{u=1}^{\infty} \sum_{p=1}^{2t}  \sum_{s=1}^{2t+1-p} 
	\mq^{4 t^2 + 8 t - 4 pt - 4 p + 2 up + 2 p^2 + 4 + 2 s (u + p)} \\
	& =  \sum_{t=1}^{\infty}  \sum_{u=1}^{\infty} \sum_{p=1}^{2t}  
	\mq^{(4t+4)(t-p+1)+2p+2u  } \, \frac{\mq^{(2u+2p)p}-\mq^{(2u+2p)(2t+1)}}{1-\mq^{2u+2p}}.
	\end{align*}
	This is the second term of the generating function in the case of $(m,n)=(0,0)$. 
	
	Suppose $i=4t+2$, we will obtain the third term. The fourth and fifth terms come from the case when $2j \leq i$.
	
	Basically, we split the terms by $4 \mid i$ or $4 \mid i+2$ when $i$ is even, and by $4 \mid i+1$ or $4 \mid i+3$ when $i$ is odd. Then the result follows from tedious calculation.
\end{proof}

\end{document}